\documentclass[11pt]{scrartcl}
\usepackage{amsfonts,amssymb,amsthm,amsmath,booktabs}
\usepackage{color,graphicx}
\usepackage[latin1]{inputenc}
\pdfoutput=1

\ifx\pdftexversion\undefined
\usepackage{nohyperref}
\else
\usepackage[colorlinks=true,linkcolor={blue},citecolor={blue},urlcolor={blue},
  pdfauthor={Thomas Apel, Serge Nicaise, Johannes Pfefferer},
  pdfstartview={Fit}]{hyperref}
\fi

\newcommand{\abssec}[1]{\noindent\normalsize {\bfseries #1\quad }\ignorespaces}
\renewenvironment{abstract}{\abssec{Abstract}}{\par\vspace{.1in}}
\newenvironment{keywords}{\abssec{Key Words}}{\par\vspace{.1in}}
\newenvironment{AMSMOS}{\abssec{AMS subject
  classification}}{\par\vspace{.1in}}

\theoremstyle{plain}
\newtheorem{theorem}{Theorem}[section]
\newtheorem{corollary}[theorem]{Corollary}

\newtheorem{lemma}[theorem]{Lemma}

\theoremstyle{definition}

\numberwithin{equation}{section}

\def\R{\mathbb{R}}

\def\S{\mbox{$\xi(r)\,r^\lambda\sin(\lambda\theta)$}}
\def\Sd{p_s}
\def\Fd{\phi_s}

\newcommand{\V}{H^1_\Delta(\Omega)\cap H^1_0(\Omega)}

\newcommand{\Span}{\text{Span}}

\definecolor{darkgreen}{rgb}{0.0, 0.5, 0.3}

\begin{document}

\title{\LARGE Adapted numerical methods
%  A dual singular complement method
  \\ for the numerical solution of the Poisson equation \\ with $L^2$
  boundary data in non-convex domains\thanks{This paper is an extension of our previous paper \cite{ApelNicaisePfefferer2014b}. The work was partially
    supported by Deutsche Forschungsgemeinschaft, IGDK 1754.}}

\author{Thomas Apel\thanks{\texttt{thomas.apel@unibw.de}, Institut
    f\"ur Mathematik und Bauinformatik, Universit\"at der Bundeswehr
    M\"unchen, D-85579 Neubiberg, Germany} \and Serge
  Nicaise\thanks{\texttt{snicaise@univ-valenciennes.fr}, LAMAV,
    Institut des Sciences et Techniques de Valenciennes, Universit\'e
    de Valenciennes et du Hainaut Cambr\'esis, B.P. 311, 59313
    Valenciennes Cedex, France } \and Johannes
  Pfefferer\thanks{\texttt{pfefferer@ma.tum.de}, Lehrstuhl f\"ur
      Optimalsteuerung, Technische Universit\"at M\"unchen,
      Boltzmannstr. 3, D-85748 Garching bei M\"unchen, Germany}}
\maketitle

\begin{abstract}
  The very weak solution of the Poisson equation with $L^2$ boundary
  data is defined by the method of transposition. The finite element
  solution with regularized boundary data converges in the
  $L^2(\Omega)$-norm with order $1/2$ in convex domains but has a
  reduced convergence order in non-convex domains although the
  solution remains to be contained in $H^{1/2}(\Omega)$. The reason is
  a singularity in the dual problem.  In this paper we propose and
  analyze, as a remedy, both a standard finite element method with
  mesh grading and a dual variant of the singular complement method.
  The error order 1/2 is retained in both cases also with non-convex
  domains.  Numerical experiments confirm the theoretical results.
\end{abstract}


\begin{AMSMOS}
  65N30; 65N15
\end{AMSMOS}

%%%%%%%%%%%%%%%%%%%%%%%%%%%%%%%%%%%%%%%%%%%%%%%%%%%%%%%%%%%%%%%%%%%%%%%%%%%%
\section{Introduction}

In this paper we consider the  boundary value problem
\begin{align} \label{eq:bvp}
  -\Delta y &= f \quad\text{in }\Omega, & 
  y &= u \quad\text{on }\Gamma=\partial\Omega,
\end{align}
with right hand side $f\in H^{-1}(\Omega)$ and boundary data $u\in L^2(\Gamma)$. We assume $\Omega\subset\R^2$
to be a bounded polygonal domain with boundary $\Gamma$.  Such
problems arise in optimal control when the Dirichlet boundary control
is considered in $L^2(\Gamma)$ only, see for example
\cite{DeckelnickGuentherHinze2009,FrenchKing1991,MayRannacherVexler2008}.
%  the papers by
% Deckelnick, G\"unther, and Hinze, \cite{DeckelnickGuentherHinze2009},
% \marginpar{\footnotesize\raggedrightdelete names or add names in other places}
% French and King, \cite{FrenchKing1991}, and May, Rannacher, and
% Vexler, \cite{MayRannacherVexler2008}.

For boundary data $u\in L^2(\Gamma)$ we cannot expect a weak solution
$y\in H^1(\Omega)$. Therefore we define a very weak solution by the
method of transposition which goes back at least to Lions and Magenes
\cite{LionsMagenes1968}: Find
\begin{align}\label{eq:veryweak2}
  y\in L^2(\Omega):\quad (y, \Delta v)_\Omega =
  (u,\partial_n v)_\Gamma - (f,v)_\Omega \quad\forall
  v\in V
\end{align}
with $(w,v)_G:=\int_G wv$ denoting the $L^2(G)$ scalar product
  or an appropriate duality product.
In our previous paper \cite{ApelNicaisePfefferer2014a} we  showed
that the appropriate space $V$ for the test functions is
\begin{align}\label{eq:H1Delta}
  V:=H^1_\Delta(\Omega)\cap H^1_0(\Omega) \quad\text{with}\quad
  H^1_\Delta(\Omega):=\{v\in H^1(\Omega): \Delta v\in L^2(\Omega)\}.
\end{align}
In particular it ensures $\partial_n v\in L^2(\Gamma)$ for $v\in V$
such that the formulation \eqref{eq:veryweak2} is well defined. We
proved the existence of a unique solution $y\in L^2(\Omega)$ for $u\in
L^2(\Gamma)$ and $f\in H^{-1}(\Omega)$, and that the solution is even in
$H^{1/2}(\Omega)$. The method of transposition is used in different
variants also in \cite{FrenchKing1991,Berggren2004,CasasRaymond2006,%
  CasasMateosRaymond2009,DeckelnickGuentherHinze2009,MayRannacherVexler2008}.

Consider now the discretization of the boundary value problem.  Let
$\mathcal{T}_h$ be a family of quasi-uniform, conforming finite
element meshes, and introduce the finite element spaces
\begin{align*}%\label{eq:discretespaces}
  Y_h := \{v_h\in H^1(\Omega): v_h|_T\in\mathcal{P}_1\ \forall
  T\in\mathcal{T}_h\}, \quad Y_{0h} := Y_h\cap H^1_0(\Omega),\quad
  Y_h^\partial := Y_h|_{\partial\Omega}.
\end{align*}
Since the boundary datum $u$ is in general not contained in
$Y_h^\partial$ we have to approximate it by $u^h\in
  Y_h^\partial$, e.\,g. by using $L^2(\Gamma)$-projection or
quasi-interpolation. In this way,
% We showed in \cite{ApelNicaisePfefferer2014a} that we can construct in
% this way a function $u^h$ with \marginpar{still there?}
% \begin{align*}%\label{eq:carstensen}
%   \|u-u^h\|_{H^{-1/2}(\Gamma)}\le Ch^{1/2}\|u\|_{L^2(\Gamma)}.
% \end{align*}
% As a side effect, 
the boundary datum is even regularized since $u^h\in
H^{1/2}(\Gamma)$. Hence we can consider a regularized (weak) solution
 in $Y_*^h:=\{v\in H^1(\Omega): v|_\Gamma=u^h\}$,  
\begin{align}\label{eq:regsol}
  y^h\in Y_*^h: \quad (\nabla y^h,\nabla v)_\Omega =
  (f,v)_\Omega \quad\forall v\in H^1_0(\Omega).
\end{align}
% We  showed in \cite{ApelNicaisePfefferer2014a} that the limit
% $y=\lim\limits_{h\to0}y^h$ belongs to $L^2(\Omega)$ and is the very
% weak solution, that means it satisfies \eqref{eq:veryweak2}.
The finite element solution $y_h$ is now searched in $Y_{*h}:=
Y_*^h\cap Y_h$ and is defined in the classical way: find
\begin{align}\label{eq:serge20/06:6}
  y_h\in Y_{*h}:\quad (\nabla y_h,\nabla v_h)_\Omega =
  (f,v_h)_\Omega\quad\forall v_h\in Y_{0h}.
\end{align}
The same discretization was derived previously by Berggren
\cite{Berggren2004} from a different point of view.
In \cite{ApelNicaisePfefferer2014a} we showed that the discretization
error estimate
\begin{align*}
  \|y-y_h\|_{L^2(\Omega)}\le Ch^s
  \left(h^{1/2}\|f\|_{H^{-1}(\Omega)}+\|u\|_{L^2(\Gamma)}\right)
\end{align*}
holds for $s=1/2$ if the domain is convex; this is a slight
improvement of the result of Berggren, and the convex case is completely treated.
  In the case of non-convex domains this convergence order is
  reduced although the very weak solution $y$ is also in
$H^{1/2}(\Omega)$; the finite element method does not lead to the best
approximation in $L^2(\Omega)$. In order to describe the result we
assume for simplicity that $\Omega$ has only one corner with interior
angle $\omega\in(\pi,2\pi)$.
% This corner is located at the origin of the coordinate system, and one
% boundary edge is contained in the positive $x_1$-axis.
%The real number $\lambda=\pi/\omega$ 
We proved in \cite{ApelNicaisePfefferer2014a} the convergence order
$s=\lambda-1/2-\varepsilon$, where $\lambda:=\pi/\omega$ and
  $\varepsilon>0$ arbitrarily small, and showed by numerical
experiments that the order of almost $\lambda-1/2$ is sharp.  Note
  that $s\to0$ for $\omega\to2\pi$. This is the state of the art for
  this kind of problem, and our aim is to devise methods to retain the
  convergence order $s=1/2$ in the non-convex case.

In order to explain the reduction in the convergence order and our
first remedy, let us first mention that we have to modify the
Aubin-Nitsche method to derive $L^2(\Omega)$-error estimates. The
first reason is that our problem has no weak solution, only the dual
problem,
\begin{align}\label{eq:defvz}
  v_z\in V:\quad (\varphi, \Delta v_z)_\Omega = (z,\varphi)_\Omega \quad\forall
  \varphi\in L^2(\Omega)
\end{align}
has. The second reason is that the solution $y$ has inhomogeneous Dirichlet data
such that an estimate of the $L^2(\Gamma)$-interpolation error of
$\partial_nv_z$ is needed. The $H^1(\Omega)$-error of a standard
finite element method is of order one in convex domains but reduces to
$s=\lambda-\varepsilon$ in the case of non-convex domains; moreover,
the order of the $L^2(\Gamma)$-interpolation error of $\partial_nv_z$
reduces from $1/2$ to $\lambda-1/2-\varepsilon$. It is known for a
long time that locally refined (graded) meshes and augmenting of the
finite element space by singular functions are appropriate to retain
the optimal convergence order for such problems, see, e.~g.,
\cite{Babuska1970,BlumDobrowolski1982,ciarletjr:03,%
  OganesjanRuhovec1979,Raugel1978,StrangFix1973}.  We use these
strategies in this paper. 

The novelty is that the adapted methods act now implicitly and
  occur essentially in the analysis for the dual problem. This sounds
  particularly simple in the case of mesh grading. However, the
  convergence proof in \cite{ApelNicaisePfefferer2014a} contains not
  only interpolation error estimates for the dual solution and its
  normal derivative (which are improved now) but also the application
  of an inverse inequality which gives a too pessimistic result if
  used unchanged in the case of graded meshes. We prove in Section
  \ref{sec:graded} a sharp result by using a weighted norm in
  intermediate steps. Note we suggest a strong mesh grading with
  grading parameter $\mu\to0$ (the parameter is explained in Section
  \ref{sec:graded}) for $\omega\to2\pi$ because of the interpolation
  error estimate of $\partial_nv_z$; the numerical tests show that
  weaker grading is not sufficient.

The basic idea of the dual singular function method, see
  \cite{BlumDobrowolski1982}, or the singular complement method, see 
  \cite{ciarletjr:03}, is to augment the approximation space for the
  solution by one (or more, if necessary) singular function of type
  $r^{\lambda}\sin(\lambda\theta)$ and the space of test functions by
  a dual function of type $r^{-\lambda}\sin(\lambda\theta)$, where
  $r,\theta$ are polar coordinates at the concave corner. In this
  paper we do it the other way round and compute an approximate
  solution
\[ z_h\in Y_h\oplus\Span\{r^{-\lambda}\sin(\lambda\theta)\},\]
such that the error estimate
\begin{align*}
  \|y-z_h\|_{L^2(\Omega)}\le Ch^{1/2}
  \left(h^{1/2}\|f\|_{H^{-1}(\Omega)}+\|u\|_{L^2(\Gamma)}\right)
\end{align*}
can be shown. Note that the original singular complement method
  augments the standard finite element space with a function which is
  part of the representation of the solution. Here, we complement the
  finite element space with $r^{-\lambda}\sin(\lambda\theta)\not\in
  H^{1/2}(\Omega)$, and although $y\in H^{1/2}(\Omega)$ this has an
  effect on the approximation order in the $L^2(\Omega)$-norm. This
  makes the method different from the original singular complement
  method, \cite{ciarletjr:03}, and we call it
  \emph{dual singular complement method}. Numerical experiments in
Section \ref{sec:experiments} confirm the theoretical results.

Finally in this introduction, we would like to note that higher order
finite elements are not useful here since the solution has low
regularity. The extension of our methods to three-dimensional domains
should be possible in the case of mesh grading (at considerable
technical expenses in the analysis) but is not straightforward in the
case of the dual singular complement method since the space
$V\setminus H^2(\Omega)$ is in general not finite dimensional, see
\cite{CiarletJungKaddouriLabrunieZou2005,CiarletJungKaddouriLabrunieZou2006}
for the Fourier singular complement method to treat special domains.
Curved boundaries could be treated at the prize of using non-affine
finite elements, see, e.~g.,
\cite{Bernardi1989,BrambleKing1994,DeckelnickGuentherHinze2009}.

\section{\label{sec:graded}Graded meshes}

Recall from the introduction that $\Omega\subset\R^2$ is a bounded
polygonal domain with boundary $\Gamma$, and we consider here the case
that $\Omega$ has exactly one corner (called \emph{singular corner})
with interior angle $\omega\in(\pi,2\pi)$. The convex case was already
treated in \cite{ApelNicaisePfefferer2014a} and the case of more than
one non-convex corners can be treated similarly since corner
singularities are local phenomena.

Without loss of generality we can assume that the singular
corner is located at the origin of the coordinate system, and that one
boundary edge is contained in the positive $x_1$-axis. We recall from
 \cite{grisvard:85a,grisvard:92b} that the weak solution of the boundary value problem 
\begin{align}\label{eq:thomas+}
  -\Delta v &= g \quad\text{in }\Omega, & 
  v &= 0 \quad\text{on }\Gamma=\partial\Omega,
\end{align}
with $g\in L^2(\Omega)$ is not contained in $H^2(\Omega)$ but in 
\begin{equation}\label{splittingreg/sing}
  \V=\left(H^2(\Omega)\cap H^1_0(\Omega)\right)\oplus \Span \{\S\},
\end{equation}
$\xi$ being a cut-off function, while $r$ and $\theta$ denote  polar coordinates at the singular corner.

Let the finite element mesh $\mathcal{T}_h=\{T\}$ be graded with
the mesh grading parameter $\mu\in(0,1]$, i.\,e., the element size
  $h_T=\mathrm{diam}\,T$ and the distance $r_T$ of the element $T$ to
  the singular corner are related by
\begin{align} \label{eq:gradingcondition}
  \begin{split}
   c_1h^{1/\mu} &\leq h_T \leq c_2h^{1/\mu} \qquad \text{ for } r_T=0, \\
   c_1h r_T^{1-\mu} &\leq h_T \leq c_2h r_T^{1-\mu} \ \quad \text{ for } r_T>0.
  \end{split}
\end{align}
Define the finite element spaces 
\begin{align}\label{eq:discretespaces}
  Y_h = \{v_h\in H^1(\Omega): v_h|_T\in\mathcal{P}_1\ \forall
  T\in\mathcal{T}_h\}, \quad Y_{0h} = Y_h\cap H^1_0(\Omega),\quad
  Y_h^\partial = Y_h|_{\partial\Omega},
\end{align}
and let the regularized boundary datum $u^h\in Y_h^\partial\subset
H^{1/2}(\Gamma)$ be defined by the $L^2(\Gamma)$-projection $\Pi_hu$
 or by the Carstensen interpolant $C_hu$, see
\cite{Carstensen1999}. To define the latter let ${\mathcal
  N}_{\Gamma}$ be the set of nodes of the triangulation on the
boundary, and set
\[
  C_hu=\sum_{x\in {\mathcal N}_{\Gamma}} \pi_x(u)\lambda_x \quad\text{with}\quad
  \pi_x(u)=\frac{\int_{\Gamma} u\lambda_x} {\int_{\Gamma} \lambda_x}
=\frac{(u,\lambda_x)_\Gamma}{(1,\lambda_x)_\Gamma},
\]
where $\lambda_x$ is the standard hat function related to $x$.
As already outlined in \cite{ApelNicaisePfefferer2014a}, the
advantages of the interpolant in comparison with the $L^2$-projection
are its local definition and the property
\[ 
  u\in[a,b]\quad\Rightarrow\quad C_hu\in[a,b],
\]
see \cite{ReyesMeyerVexler2008}; a disadvantage may be that
$C_hu_h\not=u_h$ for piecewise linear $u_h$. With these regularized
boundary data we define the regularized weak solution $y^h\in
Y_*^h:=\{v\in H^1(\Omega): v|_\Gamma=u^h\}$ by \eqref{eq:regsol}.

\begin{lemma}\label{lem:regerror}
  The effect of the regularization of the boundary datum can be estimated by
  \[
    \|y-y^h\|_{L^2(\Omega)}\le ch^{1/2}
    \left(\|u\|_{L^2(\Omega)} + h^{1/2}\|f\|_{H^{-1}(\Omega)}\right)
  \]
  if the mesh is graded with parameter $\mu<2\lambda-1$.
\end{lemma}

\begin{proof}
  In view of 
  \[
    \|y-y^h\|_{L^2(\Omega)} = \sup_{z\in L^2(\Omega), z\not=0}
    \frac{(y-y^h,z)_\Omega}{\|z\|_{L^2(\Omega)}} 
  \]
  we have to estimate $(y-y^h,z)_\Omega$. To this end,
  let $z\in L^2(\Omega)$ be an arbitrary function, let $v_z\in V$ be defined by
  \begin{align}\label{eq:defvznew}
     (\varphi, \Delta v_z)_\Omega = (z,\varphi)_\Omega \quad\forall
    \varphi\in L^2(\Omega),
  \end{align}
  see also \eqref{eq:defvz}.  Since the weak regularized solution
  $y^h\in Y_*^h:=\{v\in H^1(\Omega): v|_\Gamma=u^h\}$ defined by
  \eqref{eq:regsol} is also a very weak solution,
  \begin{align}\label{eq:yh_veryweak2}
    (y^h, \Delta v)_\Omega = (u^h,\partial_n v)_\Gamma - 
    (f,v)_\Omega \quad\forall v\in V
  \end{align}
  we get with \eqref{eq:veryweak2} and \eqref{eq:defvznew}
  \begin{align*}
    (y-y^h, z)_\Omega = (u-u^h,\partial_n v_z)_\Gamma .
  \end{align*}

  If $u^h$ is the $L^2(\Gamma)$-projection $\Pi_hu$ of $u$ we can continue with
  \begin{align*}
    (u-u^h,\partial_n v_z)_\Gamma &= 
    (u-u^h,\partial_n v_z-\Pi_h(\partial_n v_z))_\Gamma =
    (u,\partial_n v_z-\Pi_h(\partial_n v_z))_\Gamma \\ &\le
    \|u\|_{L^2(\Gamma)}\,\|\partial_n v_z-\Pi_h(\partial_n v_z)\|_{L^2(\Gamma)}  \\& \le
    \|u\|_{L^2(\Gamma)}\,\|\partial_n v_z-C_h(\partial_n v_z)\|_{L^2(\Gamma)} \\& =
    \|u\|_{L^2(\Gamma)}\,\Big\|\sum_{x\in{\mathcal N}_{\Gamma}} 
    \left( \partial_n v_z-\pi_x(\partial_n v_z) \right)\lambda_x\Big\|_{L^2(\Gamma)}\\ &\le c
    \|u\|_{L^2(\Gamma)} \Big(\sum_{x\in{\mathcal N}_{\Gamma}} 
    \|\partial_n v_z-\pi_x(\partial_n v_z)\|_{L^2(\omega_x)}^2\Big)^{1/2}.
  \end{align*}

  If  $u^h$ is the Carstensen interpolant of $u$, there holds
  \begin{align*} 
    (u-C_hu, \partial_n v_z)_\Gamma &=
    \Big(\sum_{x\in{\mathcal N}_{\Gamma}}(u-\pi_xu)\lambda_x
    , \partial_n v_z\Big)_\Gamma =
    \sum_{x\in{\mathcal N}_{\Gamma}} 
    (u-\pi_x(u), (\partial_n v_z)\lambda_x)_\Gamma \\ \nonumber &=
    \sum_{x\in{\mathcal N}_{\Gamma}} 
    (u-\pi_x(u), (\partial_n v_z-\pi_x(\partial_n v_z))\lambda_x)_\Gamma \\  &\le
    \sum_{x\in{\mathcal N}_{\Gamma}} \|u\|_{L^2(\omega_x)} 
    \|\partial_n v_z-\pi_x(\partial_n v_z)\|_{L^2(\omega_x)} \\  &\le
    c \|u\|_{L^2(\Gamma)} \Big( \sum_{x\in{\mathcal N}_{\Gamma}}  
    \|\partial_n v_z-\pi_x(\partial_n v_z)\|_{L^2(\omega_x)}^2 \Big)^{1/2},
  \end{align*}
  i.\,e., in both cases we have to estimate $\sum_{x\in{\mathcal N}_{\Gamma}}  \|\partial_n v_z-\pi_x(\partial_n v_z)\|^2_{L^2(\omega_x)}$.\pagebreak[3]
  
  To this end  we  notice that 
  \begin{align*}
    v_z\in V&=\left(H^2(\Omega)\cap H^1_0(\Omega)\right)\oplus 
    \operatorname{Span} \{\xi(r)\,r^\lambda\sin(\lambda\theta)\}, 
     \end{align*}
     and consequently
      \begin{align*}
     \partial_n v_z\in V_\Gamma &=
    \left(\prod_{j=1}^N H^{1/2}_{00}(\Gamma_j)\right)\oplus 
    \operatorname{Span} \{\xi(r)\,r^{\lambda-1}\},
  \end{align*}
  see also the discussion in~\cite{ApelNicaisePfefferer2014a}. This
  means that we can split $\partial_n v_z= \alpha\xi(r)\,
  r^{\lambda-1} + \sum_{j=1}^N w_j$ with $w_j\in H^{1/2}_{00}(\Gamma_j)$ and
  \[
    |\alpha| + \sum_{j=1}^N \|w_j\|_{H^{1/2}_{00}(\Gamma_j)}
    =: \|\partial_n v_z\|_{V_\Gamma} \leq c \|v_z\|_V := 
    \|\Delta v_z\|_{L^2(\Omega)} = \|z\|_{L^2(\Omega)}. 
  \]
  By standard estimates we obtain 
  \[
    \left( \sum_{x\in{\mathcal N}_{\Gamma}}  
    \|w_j-\pi_xw_j\|_{L^2(\omega_x)}^2 \right)^{1/2}
    \le c h^{1/2} \|w_j\|_{H^{1/2}_{00}(\Gamma_j)}
  \]
  such that it remains to show that $\left(\sum_{x\in{\mathcal N}_{\Gamma}}\|\xi(r)\,r^{\lambda-1} -
  \pi_x(\xi(r)\,r^{\lambda-1})\|_{L^2(\omega_x)}^2\right)^{1/2}\le c h^{1/2}$  to
  conclude $\left(\sum_{x\in{\mathcal N}_{\Gamma}} \|\partial_n v_z-\pi_x(\partial_n
  v_z)\|_{L^2(\omega_x)}^2\right)^{1/2}\le c h^{1/2}\|z\|_{L^2(\Omega)}$. 

Denote by ${\mathcal N}_{\Gamma,\mathrm{reg}}\subset{\mathcal
    N}_{\Gamma}$ the set of nodes where $\omega_x$ does not contain
  the singular corner. We can estimate
  \begin{align*}
    &\sum_{x\in{\mathcal N}_{\Gamma,\mathrm{reg}}} 
    \|\xi(r)\,r^{\lambda-1} - \pi_x(\xi(r)\,r^{\lambda-1})\|_{L^2(\omega_x)}^2  \le
    c \sum_{x\in{\mathcal N}_{\Gamma,\mathrm{reg}}} 
    h_x^2 \|r^{\lambda-2}\|_{L^2(\omega_x)}^2 \\ & \le
    ch \sum_{x\in{\mathcal N}_{\Gamma,\mathrm{reg}}} 
    r_x^{1-\mu}r_x \|r^{\lambda-2}\|_{L^2(\omega_x)}^2  \le
    ch \int_0^{\mathrm{diam}\Omega} r^{2-\mu+2(\lambda-2) }\mathrm{d} r =ch
  \end{align*}
  for $\mu<2\lambda-1$. For the three nodes $x\in{\mathcal N}_{\Gamma}
  \setminus {\mathcal N}_{\Gamma,\mathrm{reg}}$ we cannot use the
  $H^1(\omega_x)$-regularity of $r^{\lambda-1}$ but there holds simply
  \begin{align*}
    \|\xi(r)\,r^{\lambda-1} - \pi_x(\xi(r)\,r^{\lambda-1})\|_{L^2(\omega_x)} &\le
    c\|r^{\lambda-1}\|_{L^2(\omega_x)}  \sim
    h_x^{\lambda-1}h_x^{1/2} \sim h^{(\lambda-1/2)/\mu} \sim h^{1/2}
  \end{align*}
  for  $\mu<2\lambda-1$. This finishes the proof.
\end{proof}

We consider now a lifting $\tilde B_hu^h\in Y_{*h}$ defined by the
nodal values as follows:
\begin{align}\label{def:Rh}
  (\tilde B_hu^h)(x)&= \begin{cases} u^h(x), &\text{for all nodes } x\in\Gamma,\\
    0 &\text{for all nodes } x\in\Omega.\end{cases}
\end{align}
% Let now $\tilde B_hu^h\in Y_{*h}$ be the function that extends $u^h$
% by zero inside of $\Omega$.
The function $y^h$ and its finite element approximation $y_h\in
Y_{*h}=Y_*^h\cap Y_{h}$ are now defined by
\begin{equation}\label{eq:grad0}
  y^h=y_f+\tilde B_hu^h+\tilde y_0^h\quad\text{as well as}\quad 
  y_h=y_{fh}+\tilde B_hu^h+\tilde y_{0h},
\end{equation}
where $y_f,\tilde y_0^h\in H^1_0(\Omega)$ and $y_{fh},\tilde y_{0h}\in
Y_{0h}$ satisfy
\begin{align}
(\nabla y_f,\nabla v)_\Omega&=(f,v)_\Omega\quad \forall v\in H^1_0(\Omega),\label{eq:yf}\\
(\nabla y_{fh},\nabla v_h)_\Omega&=(f,v_h)_\Omega\quad \forall v_h\in Y_{0h}.\label{eq:yfh}\\
(\nabla \tilde y_0^h,\nabla v)_\Omega&=-(\nabla(\tilde B_h u^h),\nabla v)_\Omega\quad \forall v\in H^1_0(\Omega),\label{eq:grad1}\\
(\nabla \tilde y_{0h},\nabla v_h)_\Omega&=-(\nabla(\tilde B_h u^h),\nabla v_h)_\Omega\quad \forall v_h\in Y_{0h}.\label{eq:grad2}
\end{align}
In order to estimate $\|y^h-y_h\|_{L^2(\Omega)}$ we estimate
  $\|y_f-y_{fh}\|_{L^2(\Omega)}$ and $\|\tilde y_0^h-\tilde
  y_{0h}\|_{L^2(\Omega)}$.

\begin{lemma}\label{lem:yf-yfh}
  The error in approximating $y_f$ satisfies
  \[
    \|y_f-y_{fh}\|_{L^2(\Omega)} \le ch \|f\|_{H^{-1}(\Omega)}.
  \]
  if the mesh is graded with parameter $\mu<\lambda$. 
\end{lemma}
Note that the condition $\mu<\lambda$ is weaker than the condition
$\mu<2\lambda-1$ from Lemma~\ref{lem:regerror} since $\lambda<1$.

\begin{proof} 
    As in the proof of Lemma~\ref{lem:regerror}, let $z\in
    L^2(\Omega)$ be an arbitrary function, let $v_z\in V$ be defined
    via \eqref{eq:defvznew}, and let $v_{zh}\in Y_{0h}$ be the Ritz
    projection of $v_z$.  By the definitions \eqref{eq:yf} and
    \eqref{eq:yfh} and using the Galerkin orthogonality we get
  \begin{align*}
    (y_f-y_{fh},z)_\Omega &=
    (\nabla(y_f-y_{fh}),\nabla v_z)_\Omega =
    (\nabla(y_f-y_{fh}),\nabla (v_z-v_{zh}))_\Omega \\ &=
    (\nabla y_f,\nabla (v_z-v_{zh}))_\Omega \le
    \|\nabla y_f\|_{L^2(\Omega)} \, \|\nabla (v_z-v_{zh})\|_{L^2(\Omega)}
  \end{align*}
  By using standard a priori estimates we obtain with grading $\mu<\lambda$
  \begin{align*}
    \|\nabla y_f\|_{L^2(\Omega)} &\le \|f\|_{H^{-1}(\Omega)}, \\
    \|\nabla (v_z-v_{zh})\|_{L^2(\Omega)} &\le ch\|z\|_{L^2(\Omega)},
  \end{align*}
  and hence with 
  \[
    \|y_f-y_{fh}\|_{L^2(\Omega)} = \sup_{z\in L^2(\Omega), z\not=0}
    \frac{(y_f-y_{fh},z)_\Omega}{\|z\|_{L^2(\Omega)}} 
%   \le ch \|f\|_{H^{-1}(\Omega)}.
  \]
  the assertion of the lemma.
\end{proof}

In order to estimate $\|\tilde y_0^h-\tilde y_{0h}\|_{L^2(\Omega)}$,
we divide the domain $\Omega$ into subsets $\Omega_J$, i.e.,
\[
	\Omega = \bigcup_{J=0}^{I}\Omega_J,
\]
where $\Omega_J:=\{x:d_{J+1}\leq |x|\leq d_J\}$ for $J=1,\ldots,I-1$, $\Omega_I:=\{x:|x|\leq d_I\}$ and $\Omega_0:=\Omega\backslash\bigcup_{J=1}^I\Omega_J$. The radii $d_J$ are set to $2^{-J}$ and the index $I$ is chosen such that
\begin{equation}\label{eq:d_I}
	d_I=2^{-I}=c_Ih^{1/\mu}
\end{equation}
with a constant $c_I>1$ exactly specified later on. In addition we define the extended domains $\Omega_J'$ and $\Omega_J''$ by
\begin{align*}
	\Omega_J':=\Omega_{J-1}\cup\Omega_J\cup\Omega_{J+1}\quad\text{and}\quad\Omega_J'':=\Omega_{J-1}'\cup\Omega_J'\cup\Omega_{J+1}',
\end{align*}
respectively, with the obvious modifications for $J=0,1$ and $J=I-1,I$.

\begin{lemma}\label{lemma:grad1}
With $\sigma:=r+d_I$ there holds the estimate
\[
	\|\sigma^{(1-\mu)/2}\nabla \tilde y_0^h\|_{L^2(\Omega)}\leq c h^{-1/2}\|u\|_{L^2(\Gamma)}.
\]
\end{lemma}
\begin{proof}
We start by rearranging terms, i.e.,
\begin{align}
	\|\sigma^{(1-\mu)/2}\nabla \tilde y_0^h\|_{L^2(\Omega)}^2&=\int_{\Omega}\sigma^{1-\mu}\nabla \tilde y_0^h \cdot \nabla \tilde y_0^h\notag\\
	&=\int_{\Omega}\nabla \tilde y_0^h \cdot \nabla (\tilde y_0^h\sigma^{1-\mu})-\int_{\Omega}\tilde y_0^h\nabla \tilde y_0^h \cdot \nabla \sigma^{1-\mu}.\label{eq:grad3}
\end{align}
For the first term in \eqref{eq:grad3} we conclude according to \eqref{eq:grad1}
\begin{align}
\int_{\Omega}&\nabla \tilde y_0^h \cdot \nabla (\tilde y_0^h\sigma^{1-\mu})=-\int_{\Omega}\nabla (\tilde B_h u^h) \cdot \nabla (\tilde y_0^h\sigma^{1-\mu})\notag\\
&=-\int_{\Omega}\sigma^{1-\mu}\nabla (\tilde B_h u^h) \cdot \nabla \tilde y_0^h-\int_{\Omega}\tilde y_0^h\nabla (\tilde B_h u^h) \cdot \nabla \sigma^{1-\mu}\notag\\
&\leq \|\sigma^{(1-\mu)/2}\nabla (\tilde B_h u^h)\|_{L^2(\Omega)}\left(\|\sigma^{(1-\mu)/2}\nabla \tilde y_0^h\|_{L^2(\Omega)}+\|\sigma^{(-1-\mu)/2} \tilde y_0^h\|_{L^2(\Omega)}\right),\label{eq:grad4}
\end{align}
where we used the Cauchy-Schwarz inequality and
\begin{equation}\label{eq:nablasigma}
\nabla\sigma^{1-\mu}=(1-\mu)\sigma^{-\mu}(\cos\theta,\sin \theta)^T.
\end{equation}
Having in mind the decomposition of the domain in subdomains $\Omega_J$, an application of the Poincar\'e inequality yields for the latter term in \eqref{eq:grad4}
\begin{align*}
\|\sigma^{(-1-\mu)/2} \tilde y_0^h\|_{L^2(\Omega)}^2&=\sum_{J=0}^I\|\sigma^{(-1-\mu)/2} \tilde y_0^h\|_{L^2(\Omega_J)}\|\sigma^{(-1-\mu)/2} \tilde y_0^h\|_{L^2(\Omega_J)}\\
&\leq \sum_{J=0}^I d_J^{(-1-\mu)/2}\|\tilde y_0^h\|_{L^2(\Omega_J)}\|\sigma^{(-1-\mu)/2} \tilde y_0^h\|_{L^2(\Omega_J)}\\
&\leq c\sum_{J=0}^I d_J^{(1-\mu)/2}\|\nabla \tilde y_0^h\|_{L^2(\Omega_J')}\|\sigma^{(-1-\mu)/2} \tilde y_0^h\|_{L^2(\Omega_J)}\\
&\leq c\|\sigma^{(1-\mu)/2}\nabla \tilde y_0^h\|_{L^2(\Omega)}\|\sigma^{(-1-\mu)/2} \tilde y_0^h\|_{L^2(\Omega)},
\end{align*}
where we used $d_J\sim \sigma$ for $x\in\Omega_J'$ twice and the discrete Cauchy-Schwarz inequality.
Consequently, we get from \eqref{eq:grad4}
\begin{equation}\label{eq:grad5}
	\int_{\Omega}\nabla \tilde y_0^h \cdot \nabla (\tilde y_0^h\sigma^{1-\mu})\leq c\|\sigma^{(1-\mu)/2}\nabla (\tilde B_h u^h)\|_{L^2(\Omega)}\|\sigma^{(1-\mu)/2}\nabla \tilde y_0^h\|_{L^2(\Omega)}.
\end{equation}
Similarly to the above steps, we get for the second term in \eqref{eq:grad3} by means of \eqref{eq:nablasigma} 
\begin{align}
\int_{\Omega}&\tilde y_0^h\nabla \tilde y_0^h \cdot \nabla \sigma^{1-\mu}\leq \|\sigma^{(1-\mu)/2}\nabla \tilde y_0^h\|_{L^2(\Omega)}\|\sigma^{(-1-\mu)/2} \tilde y_0^h\|_{L^2(\Omega)}\notag\\
&\leq \|\sigma^{(1-\mu)/2}\nabla \tilde y_0^h\|_{L^2(\Omega)}\left(\|\sigma^{(-1-\mu)/2} (\tilde y_0^h+\tilde B_hu^h)\|_{L^2(\Omega)}+\|\sigma^{(-1-\mu)/2} \tilde B_hu^h\|_{L^2(\Omega)}\right),\label{eq:grad6}
\end{align}
such that we infer from \eqref{eq:grad3}, \eqref{eq:grad5} and \eqref{eq:grad6} that
\begin{align}
\|\sigma^{(1-\mu)/2}\nabla \tilde y_0^h\|_{L^2(\Omega)}&\leq c\left(\|\sigma^{(-1-\mu)/2} \tilde B_hu^h\|_{L^2(\Omega)}+\|\sigma^{(1-\mu)/2}\nabla (\tilde B_h u^h)\|_{L^2(\Omega)}\right.\notag\\
&\quad +\left.\|\sigma^{(-1-\mu)/2} (\tilde y_0^h+\tilde B_hu^h)\|_{L^2(\Omega)}\right).\label{eq:grad7}
\end{align}
Due to the definition of $\tilde B_h$ and the definition of the element size $h_T$ in case of graded meshes we easily obtain by means of the norm equivalence in finite dimensional spaces that
\begin{equation}\label{eq:grad8}
\|\sigma^{(-1-\mu)/2} \tilde B_hu^h\|_{L^2(\Omega)}
+\|\sigma^{(1-\mu)/2}\nabla (\tilde B_h u^h)\|_{L^2(\Omega)}\leq c h^{-1/2}\|u^h\|_{L^2(\Gamma)}\leq c h^{-1/2}\|u\|_{L^2(\Gamma)},
\end{equation}
where we employed the stability of $u^h$ in $L^2(\Gamma)$ in the last step. Having in mind the definition \eqref{eq:d_I} of $d_I$, we conclude by applying \cite[Lemma 2.8]{ApelNicaisePfefferer2014a} together with \cite[Remark~2.7]{ApelNicaisePfefferer2014a} that
\begin{align}
\|&\sigma^{(-1-\mu)/2} (\tilde y_0^h+\tilde B_hu^h)\|_{L^2(\Omega)}\leq d_I^{-\mu/2}\|\sigma^{-1/2} (\tilde y_0^h+\tilde B_hu^h)\|_{L^2(\Omega)}\notag\\
&\leq ch^{-1/2}\|r^{-1/2} (\tilde y_0^h+\tilde B_hu^h)\|_{L^2(\Omega)}\leq ch^{-1/2}\|u^h\|_{L^2(\Gamma)}\leq c h^{-1/2}\|u\|_{L^2(\Gamma)},\label{eq:grad9}
\end{align}
where we used again the stability of $u^h$. The estimates \eqref{eq:grad7}, \eqref{eq:grad8} and \eqref{eq:grad9} end the proof.
\end{proof}

\begin{lemma}\label{lem:yh-yh}
Let $\sigma:=r+d_I$ and $\mu\in(0,2\lambda-1)$. Then there is the estimate
\[
	\|\sigma^{-(1-\mu)/2}(\tilde y_0^h-\tilde y_{0h})\|_{L^2(\Omega)}\leq c h^{1/2}\|u\|_{L^2(\Gamma)}.
\]
\end{lemma}

\begin{proof}
Let $v\in H^1_0(\Omega)$ be the weak solution of
\[
	-\Delta v=\sigma^{-(1-\mu)}(\tilde y_0^h-\tilde y_{0h})\ \text{in } \Omega,\quad v=0\ \text{on }\Gamma,
\]
which,  according to Theorem 2.15 of \cite{BourlardDaugeLubumaNicaise90a}, has 
the regularity  
$v\in V^{2,2}_{(1-\mu)/2}(\Omega)$ (as $\mu <2\lambda-1)$ and hence $\frac12(1-\mu)>1-\lambda$)
and satisfies the a priori estimate
\begin{align}\label{eq:aprioriTA2016}
|v|_{V^{2,2}_{(1-\mu)/2}(\Omega)}\leq c\|\sigma^{-(1-\mu)}(\tilde y_0^h-\tilde y_{0h})\|_{V^{0,2}_{(1-\mu)/2}(\Omega)}\leq c \|\sigma^{-(1-\mu)/2}(\tilde y_0^h-\tilde y_{0h})\|_{L^2(\Omega)},
\end{align}
where we use the weighted Sobolev space 
$V^{k,2}_\beta(\Omega):=\{v\in\mathcal{D}':\|v\|_{V^{k,2}_\beta(\Omega)}<\infty\}$ with
\[
  \|v\|_{V^{k,2}_\beta(\Omega)}^2:=\sum_{j=1}^k |v|_{V^{j,2}_{\beta-k+j}(\Omega)}^2, \qquad
 |v|_{V^{j,2}_{\beta}(\Omega)}:=\|r^\beta \nabla^jv\|_{L^2(\Omega)}.
\]
Then we obtain by using integration by parts and the Galerkin orthogonality
\begin{align}
\lefteqn{\|\sigma^{-(1-\mu)/2}(\tilde y_0^h-\tilde y_{0h})\|_{L^2(\Omega)}^2=(\tilde y_0^h-\tilde y_{0h},-\Delta v)_\Omega}\notag\\
&=(\nabla(\tilde y_0^h-\tilde y_{0h}),\nabla(v-I_hv))_\Omega =\sum_{J=0}^I(\nabla(\tilde y_0^h-\tilde y_{0h}),\nabla(v-I_hv))_{\Omega_J}\notag\\
&\leq\sum_{J=0}^I\|\nabla(\tilde y_0^h-\tilde y_{0h})\|_{L^2(\Omega_J)}\|\nabla(v-I_hv)\|_{L^2(\Omega_J)}.\label{eq:grad10}
\end{align}

By employing standard interpolation error estimates on graded meshes we obtain for any $\mu\in(0,1]$ 
\begin{equation}\label{eq:grad11}
\|\nabla(v-I_hv)\|_{L^2(\Omega_J)}\leq chd_J^{(1-\mu)/2}|v|_{V^{2,2}_{(1-\mu)/2}(\Omega_J')},
\end{equation}
where the constant $c$  is independent of $c_I$, see e.g. \cite[Lemma 3.7]{ApelPfeffererRoesch:2012} or \cite[Lemma 3.58]{pfefferer:14}. In fact, the constant is essentially the one appearing in the local, elementwise interpolation error estimate.  Note that this kind of independence will be crucial when applying a kick back argument further below.

Local finite element error estimates from \cite[Theorem 3.4]{Demlow:2010} yield
\begin{align*}
\|\nabla(\tilde y_0^h-\tilde y_{0h})\|_{L^2(\Omega_J)}&\leq c\min_{v_h\in Y_{0h}}\left(\|\nabla(\tilde y_0^h-v_h)\|_{L^2(\Omega_J')}+\frac{1}{d_J}\|\tilde y_0^h-v_h\|_{L^2(\Omega_J')}\right)\notag\\
&\quad+c\frac{1}{d_J}\|\tilde y_0^h-\tilde y_{0h}\|_{L^2(\Omega_J')}.
\end{align*}
By choosing $v_h\equiv0$ and by applying the Poincar\'e inequality, we conclude
\begin{align}
\lefteqn{\|\nabla(\tilde y_0^h-\tilde y_{0h})\|_{L^2(\Omega_J)}\leq c\left(\|\nabla\tilde y_0^h\|_{L^2(\Omega_J')}+\frac{1}{d_J}\|\tilde y_0^h-\tilde y_{0h}\|_{L^2(\Omega_J')}\right)}\notag\\
&\leq c\left(\|\nabla\tilde y_0^h\|_{L^2(\Omega_J'')}+d_J^{(-1-\mu)/2}\|\sigma^{-(1-\mu)/2}(\tilde y_0^h-\tilde y_{0h})\|_{L^2(\Omega_J')}\right),\label{eq:grad12}
\end{align}
%\begin{align}
%\|\nabla(\tilde y_0^h-\tilde y_{0h})\|_{L^2(\Omega_J)}&\leq c(\|\nabla\tilde y_0^h\|_{L^2(\Omega_J')}+\frac{1}{d_J}\|\tilde y_0^h\|_{L^2(\Omega_J')}+\frac{1}{d_J}\|\tilde y_0^h-\tilde y_{0h}\|_{L^2(\Omega_J')})\notag\\
%&\leq c(\|\nabla\tilde y_0^h\|_{L^2(\Omega_J'')}+d_J^{(-1-\mu)/2}\|\sigma^{-(1-\mu)/2}( \tilde y_0^h-\tilde y_{0h})\|_{L^2(\Omega_J')}),\label{eq:grad12}
%\end{align}
where we used $d_J\sim \sigma$ for $x\in\Omega_J'$. Consequently, we get from \eqref{eq:grad10}--\eqref{eq:grad12}
\begin{align*}
&\|\sigma^{-(1-\mu)/2}(\tilde y_0^h-\tilde y_{0h})\|_{L^2(\Omega)}^2\notag\\
&\leq c\sum_{J=0}^I\left(h\|\sigma^{(1-\mu)/2}\nabla \tilde y_0^h\|_{L^2(\Omega_J'')}+hd_J^{-\mu}\|\sigma^{-(1-\mu)/2}( \tilde y_0^h-\tilde y_{0h})\|_{L^2(\Omega_J')}\right)|v|_{V^{2,2}_{(1-\mu)/2}(\Omega_J')}\notag\\
%&\leq c\sum_{J=0}^I(h\|\sigma^{(1-\mu)/2}\nabla \tilde y_0^h\|_{L^2(\Omega_J'')}+c_I^{-\mu}\|\sigma^{-(1-\mu)/2}(\tilde y_0^h-\tilde y_{0h})\|_{L^2(\Omega_J')})|v|_{V^{2,2}_{(1-\mu)/2}(\Omega_J')}\notag\\
&\leq c\left(h\|\sigma^{(1-\mu)/2}\nabla \tilde y_0^h\|_{L^2(\Omega)}+c_I^{-\mu}\|\sigma^{-(1-\mu)/2}(\tilde y_0^h-\tilde y_{0h})\|_{L^2(\Omega)}\right)|v|_{V^{2,2}_{(1-\mu)/2}(\Omega)},\label{eq:grad13}
\end{align*}
where we again employed $d_J\sim \sigma$ for $x\in\Omega_J''$, $hd_J^{-\mu}\leq c_I^{-\mu}$, which holds due to the definition \eqref{eq:d_I} of $d_I$, and the discrete Cauchy-Schwarz inequality. For $\mu\in(0,2\lambda-1)$ we infer by the a priori estimate \eqref{eq:aprioriTA2016}
% \[
% |v|_{V^{2,2}_{(1-\mu)/2}(\Omega)}\leq c\|\sigma^{-(1-\mu)}(\tilde y_0^h-\tilde y_{0h})\|_{V^{0,2}_{(1-\mu)/2}(\Omega)}\leq c \|\sigma^{-(1-\mu)/2}(\tilde y_0^h-\tilde y_{0h})\|_{L^2(\Omega)}
% \]
that
\[
	\|\sigma^{-(1-\mu)/2}(\tilde y_0^h-\tilde y_{0h})\|_{L^2(\Omega)}\leq c\left(h\|\sigma^{(1-\mu)/2}\nabla \tilde y_0^h\|_{L^2(\Omega)}+c_I^{-\mu}\|\sigma^{-(1-\mu)/2}(\tilde y_0^h-\tilde y_{0h})\|_{L^2(\Omega)}\right).
\]
By choosing $c_I$ large enough we can kick back the second term in the above inequality such that Lemma \ref{lemma:grad1} yields the desired result.
\end{proof}

\begin{theorem}
For $\mu\in(0,2\lambda-1)$ we get
\begin{align}\label{eq:gradingfinalestimate}
  \|y-y_h\|_{L^2(\Omega)}\leq c h^{1/2}
    \left(\|u\|_{L^2(\Omega)} + h^{1/2}\|f\|_{H^{-1}(\Omega)}\right).
\end{align}
\end{theorem}

\begin{proof}
  Due to the boundedness of $\sigma^{(1-\mu)/2}$ independent of $h$
  for all $\mu\in(0,1]$ we obtain from Lemma \ref{lem:yh-yh}
  \begin{align}
    \|\tilde y_0^h-\tilde y_{0h}\|_{L^2(\Omega)} &\le
    \|\sigma^{(1-\mu)/2}\|_{L^\infty(\Omega)}
    \|\sigma^{-(1-\mu)/2}(\tilde y_0^h-\tilde y_{0h})\|_{L^2(\Omega)}  \label{eq:estthirdterm}
    \leq c h^{1/2}\|u\|_{L^2(\Gamma)}.
  \end{align}
%   \begin{align}\nonumber
%     \|\tilde y_0^h-\tilde y_{0h}\|_{L^2(\Omega)} &\le
%     \|\sigma^{(1-\mu)/2}\|_{L^\infty(\Omega)}
%     \|\sigma^{-(1-\mu)/2}(\tilde y_0^h-\tilde y_{0h})\|_{L^2(\Omega)} \\ \label{eq:estthirdterm}
%     &\leq c h^{1/2}\|u\|_{L^2(\Gamma)}.
%   \end{align}
  In view of \eqref{eq:grad0} we get by using the triangle inequality
  \[
    \|y-y_h\|_{L^2(\Omega)} \le \|y-y^h\|_{L^2(\Omega)} + 
    \|y_f-y_{fh}\|_{L^2(\Omega)} + \|\tilde y_0^h-\tilde y_{0h}\|_{L^2(\Omega)}.
  \]
  These three terms are bounded by the right hand side of
  \eqref{eq:gradingfinalestimate} in Lemmata \ref{lem:regerror} and
  \ref{lem:yf-yfh} as well as in \eqref{eq:estthirdterm}.
\end{proof}

\section{The dual singular complement method}

\subsection{Analytical background and regularization}

Using the notation of the previous section, we recall that the splitting \eqref{splittingreg/sing}
\[
  \V=\left(H^2(\Omega)\cap H^1_0(\Omega)\right)\oplus \Span \{\S\},
\]
implies that
\[
  R:=\{\Delta v: v\in H^2(\Omega)\cap H^1_0(\Omega)\},
\]
is a closed subspace of $L^2(\Omega)$. It is shown in \cite[Sect.
2.3]{grisvard:92b} that
\begin{align}\label{eq:serge20/06:1}
  L^2(\Omega)=R \overset{\perp}{\oplus} \Span \{\Sd\},
\end{align}
with the \emph{dual singular function}
\begin{align}\label{def:ps}
  p_s=r^{-\lambda}\sin(\lambda\theta)+\tilde p_s
\end{align}
where $\tilde p_s\in H^1(\Omega)$ is chosen such that the
decomposition \eqref{eq:serge20/06:1} is orthogonal for the
$L^2(\Omega)$ inner product. Therefore, the dual singular function
$p_s$ is a solution of
\begin{align}\label{eq:thomas*}
  w\in L^2(\Omega):\quad (\Delta v,w)=0 
  \quad\forall v\in H^2(\Omega)\cap H^1_0(\Omega),
\end{align}
which proves the non-uniqueness of the solution of \eqref{eq:thomas*}.
This is the dual property to the non-existence of a solution of
\eqref{eq:thomas+} in $H^2(\Omega)\cap H^1_0(\Omega)$, see
\cite[Introduction]{grisvard:92b}.

Due to \eqref{eq:serge20/06:1} we can split any $L^2(\Omega)$-function
into $L^2(\Omega)$-orthogonal parts. To this end denote by $\Pi_R$ and
$\Pi_{\Sd}$ the orthogonal projections on $R$ and on $\Span\{\Sd\}$,
respectively, i.e., for $g\in L^2(\Omega)$, it is
$g=\Pi_Rg+\Pi_{\Sd}g$ where
\begin{align*}
  \Pi_{\Sd} g&=\alpha(g)\, \Sd \quad\text{with}\quad
  \alpha(g) =\frac{(g,\Sd)_\Omega}{\|\Sd\|_{L^2(\Omega)}^2}, \\
  \Pi_R g&=g-\Pi_{\Sd} g.
\end{align*}
Since $\Sd\in L^2(\Omega)$  there exists 
\begin{align}\label{def:phis}
\Fd\in \V:\quad -\Delta \Fd= \Sd,
\end{align}
see also Section \ref{sec:scm} for more details on $\Fd$.
% This function can be split into 
% \begin{align}\label{eq:serge20/06:11}
%  \Fd=\tilde \phi+\beta  r^\lambda \sin(\lambda \theta),
% \end{align}
% with $\tilde\phi\in H^2(\Omega)$ and
% $\beta=\pi^{-1}\|\Sd\|^2_{L^2(\Omega)}$, see \cite{ciarletjr:03}.
For the moment we assume that $\Sd$ and $\Fd$ are explicitly known;
hence the decomposition $g=\Pi_Rg+\alpha(g)\, \Sd$ can be computed
once $g$ is given. Computable approximations of $\Sd$ and $\Fd$ are
discussed in Section \ref{sec:scm}.

Now we come back to problem \eqref{eq:veryweak2} and decompose its
solution $y$ in the form
\begin{align}\label{eq:serge20/06:2}
y=\Pi_Ry+\alpha(y)\, \Sd.
\end{align}
% and $\alpha(y)\in\R$.
From the decomposition \eqref{eq:serge20/06:1} we
see that problem \eqref{eq:veryweak2} is equivalent to
\begin{align*}%\label{eq:serge20/06:3}
   (y, \Sd)_\Omega &=
  -(u,\partial_n \Fd)_\Gamma + (f,\Fd)_\Omega,
\\ %\label{eq:serge20/06:4}
 (y, \Delta v)_\Omega &=
  (u,\partial_n v)_\Gamma - (f,v)_\Omega \quad\forall
  v\in H^2(\Omega)\cap H^1_0(\Omega)
\end{align*}
and with the orthogonal splitting \eqref{eq:serge20/06:2} to
\begin{align*}%\label{eq:serge20/06:3a}
   \alpha(y)\,(\Sd, \Sd)_\Omega &=
  -(u,\partial_n \Fd)_\Gamma + (f,\Fd)_\Omega,
\\ %\label{eq:serge20/06:4a}
 (\Pi_Ry, \Delta v)_\Omega &=
  (u,\partial_n v)_\Gamma - (f,v)_\Omega \quad\forall
  v\in H^2(\Omega)\cap H^1_0(\Omega).
\end{align*}
The first equation directly yields $\alpha(y)$, namely
\begin{align}\label{eq:serge20/06:7}
  \alpha(y) =\frac{-(u,\partial_n \Fd)_\Gamma+(f,\Fd)_\Omega}%
  {(\Sd, \Sd)_\Omega},
\end{align}
hence the projection of $y$ on $\Sd$ is known. It remains to find an
approximation of $\Pi_Ry$.

At this point we recall the regularization approach from
\cite{ApelNicaisePfefferer2014a} which we summarized already in the
introduction. Let $u^h\in H^{1/2}(\Gamma)$ be a regularized boundary
datum (this can be any, e.\,g. $\Pi_hu$ or $C_hu$ from Section
\ref{sec:graded}, but we do not assume graded meshes here) such that
we can define the regularized (weak) solution in $Y_*^h:=\{v\in
H^1(\Omega): v|_\Gamma=u^h\}$,
\begin{align}\label{eq:regsol:repeat}
  y^h\in Y_*^h: \quad (\nabla y^h,\nabla v)_\Omega =
  (f,v)_\Omega \quad\forall v\in H^1_0(\Omega).
\end{align}
In \cite[Remark 2.13]{ApelNicaisePfefferer2014a} we showed that the regularization
error can be estimated by
\begin{align*}
  \|y-y^h\|_{L^2(\Omega)}\le c \|u-u^h\|_{H^{-s}(\Gamma)}
\end{align*}
where $0<s<\lambda-\frac12$ (if $\Omega$ was convex
we would get $s=\frac12$, that means
the regularization error is in general bigger in the non-convex case).
With the next lemma we show that $\Pi_R(y-y^h)$ is not affected by
non-convex corners.

\begin{lemma}\label{lem:regularizationerrornonconvex}
 The estimate
  \begin{align*}
    \|\Pi_R(y- y^h)\|_{L^2(\Omega)}\le C\|u-u^h\|_{H^{-1/2}(\Gamma)}
  \end{align*}
  holds.
\end{lemma}

\begin{proof}
  Recall $V=H^1_\Delta(\Omega)\cap H^1_0(\Omega)$ from \eqref{eq:H1Delta}.
  From \eqref{eq:regsol:repeat} and the Green formula, we have for any
  $v\in V$
  \begin{align*}
    (f,v)_\Omega = (\nabla y^h,\nabla v)_\Omega =
    -( y^h,\Delta v)_\Omega + (y^h, \partial_n v)_\Gamma .
  \end{align*}
  Note that $v\in V$ is sufficient, see  \cite[Lemma 3.4]{Costabel1988}.
  Subtracting this expression from the very weak formulation
  \eqref{eq:veryweak2}, we get
  \[
    (y-y^h,\Delta v)_\Omega=(u-u^h, \partial_n v)_\Gamma 
    \quad\forall v\in V.
  \]
  Restricting this identity to $v\in H^2(\Omega)\cap H^1_0(\Omega)$, we have
  \begin{align}\label{eq:serge20/06:5}
    ( \Pi_R(y-y^h),\Delta v)_\Omega=(u-u^h, \partial_n
    v)_\Gamma \quad \forall v\in H^2(\Omega)\cap H^1_0(\Omega).
  \end{align}
  Now for any $z\in R$, we let $v_z\in H^2(\Omega)\cap H^1_0(\Omega)$
  be the unique solution of
  \begin{align}\label{eq:sergepbDir}
    \Delta v_z=z,
  \end{align}
  that satisfies
  \begin{align}\label{eq:serge2}
    \|\partial_n v_z\|_{H^{1/2}(\Gamma)} \le 
    c\|v_z\|_{H^2(\Omega)}\le c \|z\|_{L^2(\Omega)}.
  \end{align}
  Since for any $g\in L^2(\Omega)$ the equality
  \[
    (\Pi_R(y-y^h),g)_\Omega=
    (\Pi_R(y-y^h),\Pi_Rg)_\Omega= 
    (y-y^{h},\Pi_Rg)_\Omega
  \]
  holds we get with \eqref{eq:serge20/06:5}--\eqref{eq:serge2}
  \begin{align*}
    \|\Pi_R(y-y^h)\|_{L^2(\Omega)}&= \sup_{z\in R, z\ne 0}
    \frac{(y-y^{h}, z)_\Omega}{\|z\|_{L^2(\Omega)}} =
    \sup_{z\in R, z\ne 0}
    \frac{(u-u^h,\partial_n v_z)_\Gamma}{\|z\|_{L^2(\Omega)}} \\ &\leq
    \|u-u^h\|_{H^{-1/2}(\Gamma)} \sup_{z\in R, z\ne 0}
    \frac{\|\partial_n v_z\|_{H^{1/2}(\Gamma)}}{\|z\|_{L^2(\Omega)}} \le c
    \|u-u^h\|_{H^{-1/2}(\Gamma)}
  \end{align*}
  which is the estimate to be proved.
\end{proof}

\subsection{\label{sec:3}Discretization by standard finite elements}

Recall from \eqref{eq:discretespaces} the finite element spaces
\begin{align*}
  Y_h = \{v_h\in H^1(\Omega): v_h|_T\in\mathcal{P}_1\ \forall
  T\in\mathcal{T}_h\}, \quad Y_{0h} = Y_h\cap H^1_0(\Omega),\quad
  Y_h^\partial = Y_h|_{\partial\Omega},
\end{align*}
defined now on a family $\mathcal{T}_h$ of quasi-uniform, conforming
finite element meshes. Assume that the regularized boundary datum $u^h$ is
contained in $Y_h^\partial$ such that the estimates
\begin{align}\label{eq:uhstable}
  \|u^h\|_{L^2(\Gamma)}&\le c\|u\|_{L^2(\Gamma)},\\ \label{eq:carstensen}
  \|u-u^h\|_{H^{-1/2}(\Gamma)}&\le Ch^{1/2}\|u\|_{L^2(\Gamma)},
\end{align}
hold. It can be derived from \cite[Lemma 2.14]{ApelNicaisePfefferer2014a} that this can
be accomplished by using the $L^2(\Gamma)$-projection or by
quasi-interpolation. A consequence of Lemma
\ref{lem:regularizationerrornonconvex} is the estimate
\begin{align}\label{est:regularizationerrornonconvex}
  \|\Pi_R(y- y^h)\|_{L^2(\Omega)}\le Ch^{1/2}\|u\|_{L^2(\Gamma)}.
\end{align}
(In the case of a convex
domain the operator $\Pi_R$ is the identity, and the corresponding
error estimates were already proven in
\cite{ApelNicaisePfefferer2014a}.)

As already done in the introduction, define further the finite element solution $y_h\in Y_{*h}:= Y_*^h\cap
Y_h$ via
\begin{align}\label{eq:serge20/06:6repeat}
  y_h\in Y_{*h}:\quad (\nabla y_h,\nabla v_h)_\Omega =
  (f,v_h)_\Omega\quad\forall v_h\in Y_{0h}.
\end{align}
We proved in \cite{ApelNicaisePfefferer2014a} that in the case of quasi-uniform meshes $\mathcal{T}_h$
\begin{align}
  \|y-y_h\|_{L^2(\Omega)}\le Ch^s
  \left(h^{1/2}\|f\|_{H^{-1}(\Omega)}+\|u\|_{L^2(\Gamma)}\right)\label{jonny:fe_error}
\end{align}
holds for $s\in(0,\lambda-\frac12)$ (again $s=\frac12$ for convex domains).  As before, in the next lemma
we show that $\Pi_R(y-y_h)$ is not affected by the non-convex corners.

\begin{lemma}\label{lem:3.1}
The discretization error estimate
  \begin{align*}
    \|\Pi_R(y-y_h)\|_{L^2(\Omega)}\le C  h^{1/2} 
    \left(h^{1/2}\|f\|_{H^{-1}(\Omega)}+ \|u\|_{L^2(\Gamma)}\right)
  \end{align*}
  holds. 
\end{lemma}

\begin{proof}
  By the triangle inequality we have
  \begin{align} \label{thomas:triangleinequality0207}
    \|\Pi_R(y-y_h)\|_{L^2(\Omega)} &\le 
    \|\Pi_R(y-y^h)\|_{L^2(\Omega)} + \|\Pi_R(y^h-y_h)\|_{L^2(\Omega)}. 
  \end{align}
  The first term is estimated in
  \eqref{est:regularizationerrornonconvex}.  For the second term we
  first notice that $y^h-y_h\in H^1_0(\Omega)$ satisfies the Galerkin
  orthogonality
  \begin{align}\label{eq:y^h-y_hGalerkin}
    (\nabla (y^h-y_h),\nabla v_h)_\Omega =0\quad\forall v_h\in
    Y_{0h},
  \end{align}
  see \eqref{eq:regsol} and \eqref{eq:serge20/06:6}.
  With that, we estimate $\|\Pi_R(y^h-y_h)\|_{L^2(\Omega)}$ by a
  similar arguments as $\|\Pi_R(y-y^h)\|_{L^2(\Omega)}$ in the proof
  of Lemma \ref{lem:regularizationerrornonconvex}. Recall from
  \eqref{eq:sergepbDir} and \eqref{eq:serge2} that $v_z\in
  H^2(\Omega)\cap H^1_0(\Omega)$ is the weak solution of $\Delta v_z=z\in
  R$. It can be approximated by the Lagrange interpolant $I_hv_z$
  satisfying
  \[
    \|\nabla(v_z-I_h v_z)\|_{L^2(\Omega)}\le ch \|v_z\|_{H^2(\Omega)} 
    \le ch \|z\|_{L^2(\Omega)}.
  \]
  We get 
  \begin{align}
    \|\Pi_R(y^h-y_h)\|_{L^2(\Omega)}&=\sup_{z\in R, z\ne 0}
    \frac{(y^h-y_h, z)_\Omega}{\|z\|_{L^2(\Omega)}} =
    \sup_{z\in R, z\ne 0} \frac{(\nabla(y^h-y_h), 
    \nabla v_z)_\Omega}{\|z\|_{L^2(\Omega)}} \nonumber \\ &=
    \sup_{z\in R, z\ne 0} \frac{(\nabla(y^h-y_h), 
    \nabla (v_z-I_h v_z))_\Omega}{\|z\|_{L^2(\Omega)}} \nonumber  \\&\le 
    ch \|\nabla(y^h-y_h)\|_{L^2(\Omega)}.
    \label{eq:thomas0207}
  \end{align} 

  In order to bound $\|\nabla(y^h-y_h)\|_{L^2(\Omega)}$ by the data we
  consider the lifting $\tilde B_hu^h\in Y_{*h}$ defined by
  \eqref{def:Rh}.  The next steps are simpler than in Section
    \ref{sec:graded} since we have quasi-uniform meshes and obtain a
    sharp estimate also by using an inverse inequality below.
%   the nodal values as follows:
%   \begin{align}\label{def:Rh}
%     (\tilde B_hu^h)(x)&= \begin{cases} u^h(x), &\text{for all nodes } x\in\Gamma,\\
%     0 &\text{for all nodes } x\in\Omega.\end{cases}
%   \end{align}
  The homogenized solution $y_0^h=y^h-\tilde B_hu^h\in H^1_0(\Omega)$ satisfies
  \begin{align*}%\label{eq:serge5}
    (\nabla y_0^h,\nabla v)_\Omega = (f,v)_\Omega - 
    (\nabla (\tilde B_hu^h),\nabla v)_\Omega \quad\forall v\in H^1_0(\Omega).
  \end{align*}
  By taking $v=y_0^h$ we see that
  \begin{align*}
    \|\nabla y_0^h\|^2_{L^2(\Omega)} \le \|f\|_{H^{-1}(\Omega)}\|y_0^h\|_{H^1(\Omega)}
    +\|\nabla (\tilde B_hu^h)\|_{L^2(\Omega)} \|\nabla y_0^h\|_{L^2(\Omega)}.
  \end{align*}
  Using the Poincar\'e inequality we obtain
  \begin{align}\label{eq:serge7}
    \|\nabla y_0^h\|_{L^2(\Omega)} \le c\|f\|_{H^{-1}(\Omega)}+
    \|\nabla (\tilde B_hu^h)\|_{L^2(\Omega)},
  \end{align}
  and with the C\'ea lemma
  \begin{align*}
    \|\nabla(y^h-y_h)\|_{L^2(\Omega)} &\le 
    \|\nabla(y^h-\tilde B_hu^h)\|_{L^2(\Omega)} = 
    \|\nabla y_0^h\|_{L^2(\Omega)} \\ &\le c\|f\|_{H^{-1}(\Omega)}+
    \|\nabla (\tilde B_hu^h)\|_{L^2(\Omega)}.
  \end{align*}

  The remaining term $\|\nabla (\tilde B_hu^h)\|_{L^2(\Omega)}$
  is estimated by using the inverse inequality
  \begin{align*}%\label{eq:serge4}
    \|\nabla (\tilde B_hu^h)\|_{L^2(T)}\le ch^{-1/2} \|u^h\|_{L^2(E)}.
  \end{align*}
  for $E\subset T\cap\Gamma$, $T\in\mathcal{T}_h$, which can
  be proved by standard scaling arguments, to get
  \begin{align}\label{eq:serge3}
    \|\nabla (\tilde B_hu^h)\|_{L^2(\Omega)}\le ch^{-1/2} \|u^h\|_{L^2(\Gamma)}.
  \end{align}
  Hence we proved
  \begin{align*}
    \|\nabla (y^h-y_h)\|_{L^2(\Omega)}  &\le
    c\|f\|_{H^{-1}(\Omega)} + ch^{-1/2} \|u^h\|_{L^2(\Gamma)}.
  \end{align*}
  With \eqref{thomas:triangleinequality0207}, \eqref{est:regularizationerrornonconvex},
  \eqref{eq:thomas0207}, the previous inequality, and \eqref{eq:uhstable} we finish the proof.
\end{proof}

With \eqref{eq:serge20/06:2} we can immediately conclude the following result.

\begin{corollary}\label{cor:serge}
 Let $y_h\in Y_{*h}$ be the
  solution of \eqref {eq:serge20/06:6repeat}, then the discretization error
  estimate
  \begin{align*}
    \|y-(\Pi_R y_h+\alpha(y) \Sd)\|_{L^2(\Omega)}\le Ch^{1/2}
    \left(h^{1/2}\|f\|_{H^{-1}(\Omega)}+ \|u\|_{L^2(\Gamma)}\right)
  \end{align*}
  holds, reminding that $p_s$ and $\alpha(y)$ are given by \eqref{def:ps} and \eqref{eq:serge20/06:7}, respectively.
\end{corollary}

Hence the positive result is that $\Pi_R y_h+\alpha(y)\Sd$ is a better
approximation of $y$ than $y_h$. The problem is that $\Sd$ and $\Fd$
are used explicitly, and in practice they are not known. A remedy of
this drawback is the aim of the next section.

\subsection{\label{sec:scm}Approximate singular functions}

Following \cite{ciarletjr:03}, we approximate $p_s$ from \eqref{def:ps} by 
\begin{align}\label{eq:psh}
  \begin{split}
    p_s^h &= p_h^* - r_h + r^{-\lambda}\sin(\lambda\theta), \quad
    r_h = \tilde B_h \left( r^{-\lambda}\sin(\lambda\theta)\right), \\
    p_h^*&\in Y_{0h}:\quad (\nabla p_h^*,\nabla v_h)_\Omega =
    (\nabla r_h,\nabla v_h)_\Omega \quad\forall v_h\in Y_{0h},
  \end{split}
\end{align}
with $\tilde B_h$ from \eqref{def:Rh}. 
The function $\Fd$ from \eqref{def:phis} admits the splitting
\begin{align}\label{eq:serge20/06:11}
  \Fd=\tilde \phi+\beta r^\lambda \sin(\lambda \theta),
\end{align}
with $\tilde \phi\in H^2(\Omega)$ and
$\beta=\pi^{-1}\|\Sd\|^2_{L^2(\Omega)}$, see again
\cite{ciarletjr:03}.  It is approximated by
\begin{align}\label{def:phish}
  \begin{split}
    \phi_s^h &= \phi_h^*-\beta_hs_h+\beta_h r^\lambda\sin(\lambda\theta), \quad
    s_h = \tilde B_h \left( r^\lambda\sin(\lambda\theta)\right), \quad
    \beta_h=\frac1\pi\|\Sd^h\|^2_{L^2(\Omega)}, \\
    \phi_h^*&\in Y_{0h}:\quad (\nabla \phi_h^*,\nabla v_h)_\Omega =
    (p_s^h,v_h)_\Omega + 
    \beta_h (\nabla s_h,\nabla v_h)_\Omega \quad\forall v_h\in Y_{0h},
  \end{split}
\end{align}
that means, $\tilde\phi$ is approximated by
$\tilde\phi_h=\phi_h^*-\beta_hs_h\in Y_h$.
The approximation errors are bounded by
\begin{align}
  \label{eq:serge20/06:10}
  \|\Sd-\Sd^h\|_{L^2(\Omega)}&\le ch^{2\lambda-\epsilon}\le ch, \\
  \label{eq:serge20/06:12}
  |\beta-\beta_h|&\le ch^{2\lambda-\varepsilon}\le ch, \\
  \label{eq:serge20/06:13}
  \|\Fd-\Fd^h\|_{1,\Omega}&\le ch,
\end{align}
see \cite[Lemmas 3.1--3.3]{ciarletjr:03}, where
%Remark that this estimate implies that
%\begin{align}\label{eq:serge20/06:10.5}
%|\|\Sd\|_{L^2(\Omega)}^2-\|\Sd^h\|_{L^2(\Omega)}^2|\lesssim h,
%\end{align}
\eqref{eq:serge20/06:12} and \eqref{eq:serge20/06:13} imply
\begin{align}\label{eq:serge20/06:13b}
\|\tilde \phi -\tilde \phi_h\|_{1,\Omega}\le ch.
\end{align}
% Note that due to the definition of $p_s$ we have $\|p_s\|_{L^2(\Omega)}=c$ and
% \begin{equation}\label{jonny:ps_psh}
%   \|p_s\|_{L^2(\Omega)}-ch\leq \|p_s^h\|_{L^2(\Omega)} \leq \|p_s\|_{L^2(\Omega)}+ch
% \end{equation}
% according to \eqref{eq:serge20/06:10}.

At the end of Section \ref{sec:3} we saw that $\Pi_R y_h+\alpha(y)\Sd$
is a better approximation of $y$ than $y_h$. Since this function is
not computable we approximate it by 
\begin{align}\label{eq:serge20/06:14}
  z_h=\Pi_R^h\, y_h+\alpha_h \Sd^h,
\end{align}
with
\begin{align}\label{def:PiRhyh}
  \Pi_R^h\, y_h=y_h-\gamma_h
  \Sd^h, \quad \gamma_h=\frac{(y_h, \Sd^h)_\Omega}{\|\Sd^h\|^2_{L^2(\Omega)}}
\end{align}
and a suitable approximation $\alpha_h$ of 
\[
  \alpha(y) =\frac{-(u,\partial_n \Fd)_\Gamma+(f,\Fd)_\Omega}%
  {(\Sd, \Sd)_\Omega}
\]
from \eqref{eq:serge20/06:7}.  To this end we write the problematic
term by using \eqref{eq:serge20/06:11} as
\[
  (u,\partial_n \Fd)_\Gamma=(u,\partial_n \tilde \phi)_\Gamma+
\beta (u,\partial_n (r^\lambda \sin(\lambda \theta)))_\Gamma.
\]
and replace the term $(u,\partial_n \tilde \phi)_\Gamma$ by
$(u^h,\partial_n \tilde \phi)_\Gamma$. Since $\tilde \phi$
belongs to $H^2(\Omega)$ and $u^h$ is the trace of $\tilde B_hu^h$, we get by
using the Green formula
\begin{align}\nonumber
  (u^h,\partial_n \tilde \phi)_\Gamma&=
  (\tilde B_h u^h,\Delta \tilde \phi)_\Omega +
  (\nabla  \tilde B_h u^h,\nabla \tilde \phi)_\Omega \\ \label{eq:thomas0307a} &=
  - (\tilde B_h u^h,\Sd)_\Omega +
  (\nabla  \tilde B_h u^h, \nabla \tilde \phi)_\Omega
\end{align}
as $\Delta \tilde \phi=\Delta \Fd=-\Sd$.
With all these notations and results, we define
\begin{align}\label{eq:serge20/06:15}
  \alpha_h = \frac{
    (\tilde B_h u^h,\Sd^h)_\Omega -
    (\nabla  \tilde B_h u^h, \nabla \tilde \phi_h)_\Omega
    - \beta_h(u,\partial_n (r^\lambda \sin(\lambda \theta)))_\Gamma
    + (f,\Fd^h)_\Omega
  }{(\Sd^h,\Sd^h)_\Omega^2}.
\end{align}
Note that $\alpha_h$ can be computed explicitly and therefore $z_h$ as well.

Let us estimate the approximation errors made.
\begin{lemma}\label{lem:4.1}
Let $y_h\in Y_{*h}$ be the
  solution of \eqref {eq:serge20/06:6repeat}. Then the error estimates
  \begin{align}\label{eq:serge20/06:17a}
    \|\Pi_R y_h-\Pi_R^h\, y_h\|_{L^2(\Omega)} &\le
    ch  \left(\|f\|_{H^{-1}(\Omega)}+ \|u\|_{L^2(\Gamma)}\right), \\ 
    \label{eq:serge20/06:20}
    |\alpha(y) -\alpha_h| &\le 
    ch^{1/2}  \left(h^{1/2}\|f\|_{H^{-1}(\Omega)}+ \|u\|_{L^2(\Gamma)}\right)
  \end{align}
  hold.
\end{lemma}

\begin{proof}
  With the definitions of $\Pi_R$ and $\Pi_R^h$, with
  $\gamma:=(y_h,\Sd)_\Omega/\|\Sd\|_{L^2(\Omega)}^2$, and by
  using the triangle inequality we have
  \begin{align*}
    \|\Pi_R y_h-\Pi_R^h\, y_h\|_{L^2(\Omega)} &=
    \|\gamma p_s-\gamma_h p_s^h\|_{L^2(\Omega)} \\ &\le
    |\gamma-\gamma_h|\,\|p_s^h\|_{L^2(\Omega)} + 
    |\gamma|\,\|p_s-p_s^h\|_{L^2(\Omega)}
  \end{align*}
%   To prove \eqref{eq:serge20/06:17a},  it suffices to show that
%   \begin{align}\label{eq:serge20/06:19}
%     |\gamma -\gamma_h |\le Ch^{1/2}  (\|f\|_{H^{-1}(\Omega)}+ \|u\|_{L^2(\Gamma)}),
%   \end{align}
  We write
  \begin{align*}
    \gamma-\gamma_h &= \frac{(y_h, \Sd)_\Omega}{\|\Sd\|^2_{L^2(\Omega)} } -
    \frac{(y_h, \Sd^h)_\Omega}{\|\Sd^h\|^2_{L^2(\Omega)} } \\
    &=\frac{(y_h, \Sd-\Sd^h)_\Omega}{\|\Sd\|^2_{L^2(\Omega)} }+ (y_h, \Sd^h)_\Omega
    \left(\frac{1}{\|\Sd\|^2_{L^2(\Omega)}
    }-\frac{1}{\|\Sd^h\|^2_{L^2(\Omega)} }\right)\\
    &=\frac{(y_h, \Sd-\Sd^h)_\Omega}{\|\Sd\|^2_{L^2(\Omega)} }+ (y_h, \Sd^h)_\Omega
    \frac{(p_s^h+p_s,p_s^h-p_s)_{\Omega}}{\|\Sd\|^2_{L^2(\Omega)}\|\Sd^h\|^2_{L^2(\Omega)}},
  \end{align*}
  and by the Cauchy-Schwarz inequality and \eqref{eq:serge20/06:10} we get
  \[
    |\gamma-\gamma_h|\le ch \|y_h\|_{L^2(\Omega)}.
  \]
  We have used that $\|p_s\|_{L^2(\Omega)}$ and $\|p_s^h\|_{L^2(\Omega)}$ can be treated as constants due to the definition of $p_s$ and due to \eqref{eq:serge20/06:10}.
  We conclude with $|\gamma|\le
  c \|y_h\|_{L^2(\Omega)}$, and \eqref{eq:serge20/06:10} that
  \begin{align}
    \|\Pi_R y_h-\Pi_R^h\, y_h\|_{L^2(\Omega)} &\le
   ch \|y_h\|_{L^2(\Omega)}.\label{jonny:Pi_R}
  \end{align}
  In view of the finite element error estimate \eqref{jonny:fe_error} and the standard a priori estimate for the very weak solution,
  \begin{equation*}
    \|y\|_{L^2(\Omega)}\leq c\left(\|f\|_{H^{-1}(\Omega)}+\|u\|_{L^2(\Gamma)}\right),
  \end{equation*}
  see Lemma 2.3 of \cite{ApelNicaisePfefferer2014a}, we have
  \[
    \|y_h\|_{L^2(\Omega)}\leq \|y\|_{L^2(\Omega)}+\|y-y_h\|_{L^2(\Omega)}\leq c\left(\|f\|_{H^{-1}(\Omega)}+\|u\|_{L^2(\Gamma)}\right).
  \]
%   As $y^h=\tilde B_hu^h+y_0^h$, by \eqref{eq:serge3}, the Poincar\'e
%   inequality, and \eqref{eq:serge7} we have
%   \begin{align*}%\label{eq:serge20/06:21}
%     \|y^h\|_{L^2(\Omega)} \le c \|f\|_{H^{-1}(\Omega)}+ ch^{-1/2}\|
%     u\|_{L^2(\Gamma)}.
%   \end{align*}
  This estimate together with \eqref{jonny:Pi_R} proves \eqref{eq:serge20/06:17a}.
\pagebreak[3]

  The proof of the estimate \eqref{eq:serge20/06:20} is based on
  writing the problematic term in the definition of $\alpha(y)$
  without approximation as
%   \begin{align*}
%     \alpha(y) &= \frac{-(u,\partial_n \Fd)_\Gamma+(f,\Fd)_\Omega}%
%     {(\Sd, \Sd)_\Omega} \\&=
%     \frac{-(u,\partial_n \tilde \phi)_\Gamma-
%            \beta (u,\partial_n (r^\lambda \sin(\lambda \theta)))_\Gamma
%            +(f,\Fd)_\Omega}{(\Sd, \Sd)_\Omega} \\&=
%     %
%   \end{align*}
  \begin{align*}
    (u,\partial_n \Fd)_\Gamma &=
    (u,\partial_n \tilde \phi)_\Gamma+
    \beta (u,\partial_n (r^\lambda \sin(\lambda \theta)))_\Gamma \\ &=
    (u-u^h,\partial_n \tilde \phi)_\Gamma+
    (u^h,\partial_n \tilde \phi)_\Gamma+
    \beta (u,\partial_n (r^\lambda \sin(\lambda \theta)))_\Gamma \\ &=
    (u-u^h,\partial_n \tilde \phi)_\Gamma
      - (\tilde B_h u^h,\Sd)_\Omega +
  (\nabla  \tilde B_h u^h, \nabla \tilde \phi)_\Omega+ %\\ &\qquad
    \beta (u,\partial_n (r^\lambda \sin(\lambda \theta)))_\Gamma 
  \end{align*}
  where we used \eqref{eq:thomas0307a} in the last step. Consequently,
  we showed that
  \begin{align*}
    \alpha(y)-\alpha_h& =\frac{1}{\|\Sd\|_{L^2(\Omega)}^2} \Big(
    -(u-u^h,\partial_n \tilde \phi)_\Gamma +
     (\tilde B_h u^h,\Sd-\Sd^h)_\Omega  -
    (\nabla  \tilde B_h u^h, \nabla(\tilde\phi-\tilde\phi_h))_\Omega \\ &\qquad
     -   (\beta-\beta_h)\, 
    (u,\partial_n (r^\lambda \sin(\lambda \theta)))_\Gamma + %\\ &\qquad 
    (f,\Fd-\Fd^h)_\Omega\Big).
  \end{align*}
  To prove \eqref{eq:serge20/06:20}, in view of
  \eqref{eq:serge20/06:10}, \eqref{eq:serge20/06:12}, and
  \eqref{eq:serge20/06:13} it remains to show that
  \begin{align*}
    \left|(u-u^h,\partial_n \tilde \phi)_\Gamma\right| 
    &\le ch^{1/2}\|u\|_{L^2(\Gamma)},\\
    \left|(\tilde B_h u^h, \Sd-\Sd^h)_\Omega\right| 
    &\le ch^{1/2}\|u\|_{L^2(\Gamma)}, \\
    \left|(\nabla \tilde B_h u^h,\nabla(\tilde \phi-\tilde \phi^h))_\Omega\right|
    &\le ch^{1/2}\|u\|_{L^2(\Gamma)}.
  \end{align*}
  The first estimate follows from the estimate \eqref{eq:carstensen}
  and the fact that $\tilde \phi$ belongs to $H^2(\Omega)$.  The
  second one follows from the Cauchy-Schwarz inequality and the
  estimates \eqref{eq:serge3} and \eqref{eq:serge20/06:10}. Similarly,
  the third estimate follows from the Cauchy-Schwarz inequality and the
  estimates \eqref{eq:serge3} and \eqref{eq:serge20/06:13b}.
\end{proof}

\begin{corollary}
  Let $\Omega$ be a non-convex domain and let $y_h\in Y_{*h}$ be the
  solution of \eqref {eq:serge20/06:6repeat} and let $z_h$ be derived
  by \eqref{eq:serge20/06:14}, \eqref{def:PiRhyh}, and
  \eqref{eq:serge20/06:15}, then the discretization error estimate
  \begin{align*}
    \|y-z_h\|_{L^2(\Omega)}\le Ch^{1/2}  
    \left(h^{1/2}\|f\|_{H^{-1}(\Omega)}+ \|u\|_{L^2(\Gamma)}\right)
  \end{align*}
  holds.
\end{corollary}

\begin{proof}
  The main ingredients of the proof were already derived. Indeed, it is
  \begin{align*}
    \|y-z_h\|_{L^2(\Omega)}&=
    \|\Pi_Ry+\alpha(y) \Sd-\Pi_R^h\, y_h-\alpha_h \Sd^h \|_{L^2(\Omega)} \\ &\le
    \|\Pi_Ry-\Pi_Ry_h\|_{L^2(\Omega)} +
    \|\Pi_Ry_h-\Pi_R^h\, y_h\|_{L^2(\Omega)} + \\ &\qquad
    |\alpha(y)-\alpha_h| \,\| \Sd \|_{L^2(\Omega)} +
    |\alpha_h|\, \| \Sd- \Sd^h \|_{L^2(\Omega)}.
  \end{align*}
  
  The first three terms can be estimated by using Lemmas \ref{lem:3.1}
  and \ref{lem:4.1}. So it remains to treat the fourth term.
  To bound $|\alpha_h|$ we use the triangle inequality
  \[
    |\alpha_h|\leq |\alpha_h-\alpha(y)|+|\alpha(y)|.
  \]
  For the first term we use \eqref{eq:serge20/06:20}, while for the
  second term we use \eqref{eq:serge20/06:7} reminding that $\phi_s$
  belongs to $H^{3/2+\epsilon}(\Omega)$ with some $\epsilon>0$. 
  Altogether we have 
  \[
    |\alpha_h|\leq C
    \left(\|f\|_{H^{-1}(\Omega)}+ \|u\|_{L^2(\Gamma)}\right)
  \]
  and conclude by using \eqref{eq:serge20/06:10}.
\end{proof}

\subsection{The method in form of an algorithm}

Before we describe the numerical experiments, let us summarize the
algorithm.
\begin{enumerate}
\item Compute the finite element solution
  \begin{align*}
    y_h\in Y_{*h}:\quad (\nabla y_h,\nabla v_h)_\Omega =
    (f,v_h)_\Omega\quad\forall v_h\in Y_{0h}
  \end{align*}
  where $Y_{*h}=\{v_h\in Y_h: v_h|_\Gamma=u^h\}$, compare
  \eqref{eq:serge20/06:6}, with $u^h\in Y_h^\partial$ being an
  approximation of the boundary datum $u$ satisfying
  \eqref{eq:uhstable} and \eqref{eq:carstensen}. 
\item Compute the approximate singular functions:
  \begin{align*}
    r_h &= \tilde B_h \left( r^{-\lambda}\sin(\lambda\theta)\right), \\
    p_h^*&\in Y_{0h}:\quad (\nabla p_h^*,\nabla v_h)_\Omega =
    (\nabla r_h,\nabla v_h)_\Omega \quad\forall v_h\in Y_{0h}, \\
    \tilde p_h&= p_h^* - r_h, \\
    \beta_h&=\frac1\pi
    \|\tilde p_h + r^{-\lambda}\sin(\lambda\theta)\|^2_{L^2(\Omega)}, \\
    s_h &= \tilde B_h \left( r^\lambda\sin(\lambda\theta)\right), \\
    \phi_h^*&\in Y_{0h}:\quad (\nabla \phi_h^*,\nabla v_h)_\Omega =
    ( \tilde p_h + r^{-\lambda}\sin(\lambda\theta),v_h)_\Omega +  
    \beta_h (\nabla s_h,\nabla v_h)_\Omega \quad\forall v_h\in Y_{0h}, \\
    \tilde\phi_h &= \phi_h^*-\beta_hs_h, 
  \end{align*}
%   and remember that
%   $p_s^h=\tilde p_h+r^{-\lambda}\sin(\lambda\theta)$ and
%   $\phi_s^h=\tilde\phi_h+\beta_h r^\lambda\sin(\lambda\theta)$,
  compare \eqref{eq:psh} and \eqref{def:phish}.
\item Compute 
  \begin{align*}
    \gamma_h &= \frac{(y_h, \Sd^h)_\Omega}{(\Sd^h,\Sd^h)_\Omega} 
    \quad\text{with }p_s^h=\tilde p_h+r^{-\lambda}\sin(\lambda\theta),\\
    \alpha_h &=\frac{
    (\tilde B_h u^h,\Sd^h)_\Omega -
    (\nabla  \tilde B_h u^h, \nabla \tilde \phi_h)_\Omega
    - \beta_h(u,\partial_n (r^\lambda \sin(\lambda \theta)))_\Gamma
    + (f,\Fd^h)_\Omega
  }{(\Sd^h,\Sd^h)_\Omega^2}, \\
    \delta_h&=\alpha_h-\gamma_h, \\
    \tilde  z_h&=y_h+\delta_h\tilde p_h,
  \end{align*}
  compare \eqref{def:PiRhyh} and \eqref{eq:serge20/06:15}. According
  to \eqref{eq:serge20/06:14}, the numerical solution is
  \begin{align*}
    z_h&=\tilde z_h + \delta_hr^{-\lambda}\sin(\lambda\theta).
  \end{align*}
\end{enumerate}\pagebreak[3]
  Note that all integrals with $r^\lambda$ and $r^{-\lambda}$ must be
  computed with care.

\section{\label{sec:experiments}Numerical experiment}

This section is devoted to the numerical verification of our
theoretical results.  For that purpose we present an example with known
solution. Furthermore, to examine the influence of the corner
singularities, we consider several polygonal domain $\Omega_\omega$
depending on an interior angle $\omega\in(0,2\pi)$; we present here the results for $\omega=270^\circ$ and $\omega=355^\circ$. The computational
domains are defined by
\begin{equation}\label{eq:compdomain}
  \Omega_\omega:=(-1,1)^2\cap
  \{x\in \R^2: (r(x),\theta(x))\in(0,\sqrt{2}]\times[0,\omega]\},
\end{equation}
where $r$ and $\theta$ stand for the polar coordinates located at the
origin. The boundary of $\Omega_\omega$ is denoted by $\Gamma_\omega$%
% which is decomposed into straight line segments $\Gamma_j$,
% $j=1,\ldots,m(\omega)$, counting counterclockwise beginning at the
% origin
.  
We solve the problem
% \begin{equation}\label{eq:linexample}
%   \begin{aligned}
%     -\Delta y   &= 0 && \text{in }\Omega_\omega,\\
%     y &= g && \text{on }\Gamma_j, \quad j=1,\dots,m,
%   \end{aligned}
% \end{equation}
\begin{align}\label{eq:linexample}
    -\Delta y   &= 0 \quad \text{in }\Omega_\omega, &
    y &= u \quad \text{on }\Gamma_\omega,
\end{align}
numerically by using a standard finite element method with graded
meshes and the proposed dual singular function method with
quasi-uniform meshes. The boundary datum $u$ is chosen to be
\[
  u:=r^{-0.4999}\sin(-0.4999\,\theta)\quad \text{on } \Gamma_\omega.
\]
This function belongs to $L^p(\Gamma)$ for every $p<2.0004$. The exact
solution of our problem is simply
\[
  y=r^{-0.4999}\sin(-0.4999\,\theta),
\]
since $y$ is harmonic.

Quasi-uniform finite element meshes are generated from a coarse
initial mesh by using a newest vertex bisection algorithm. Graded
meshes are generated by marking and bisecting elements until the
grading condition \eqref{eq:gradingcondition} is fulfilled with
suitable constants $c_1$ and $c_2$, see Figure \ref{fig:meshes}.  As a
regularization we have used the $L^2(\Gamma)$-projection. The
discretization errors are calculated by an adaptive quadrature
formula.

\begin{figure}
  \hspace*{\fill}\begin{minipage}{0.3\textwidth}
    \includegraphics[width=\textwidth]%
    {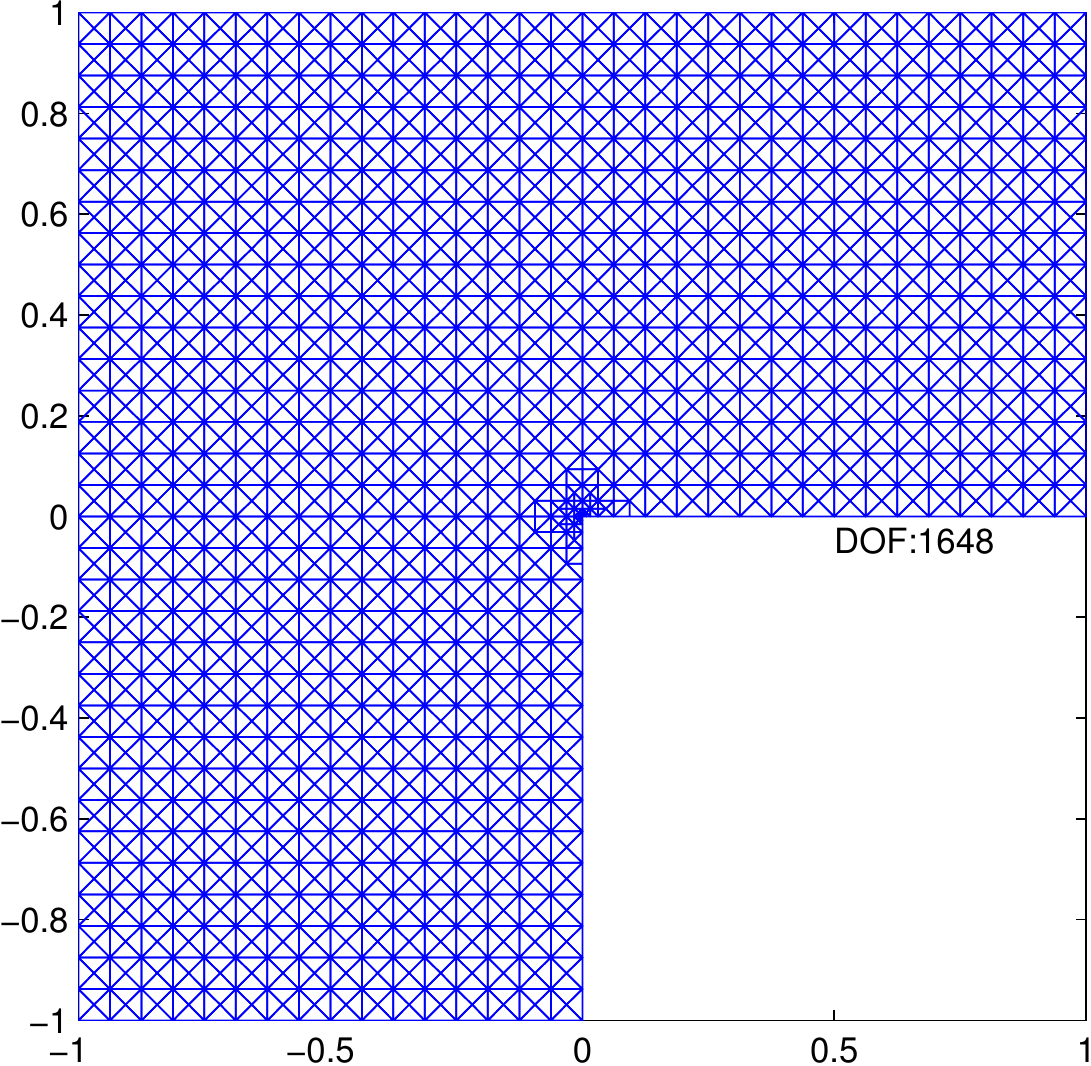}
  \end{minipage}
  \hfill
  \begin{minipage}{0.3\textwidth}
    \includegraphics[width=\textwidth]%
    {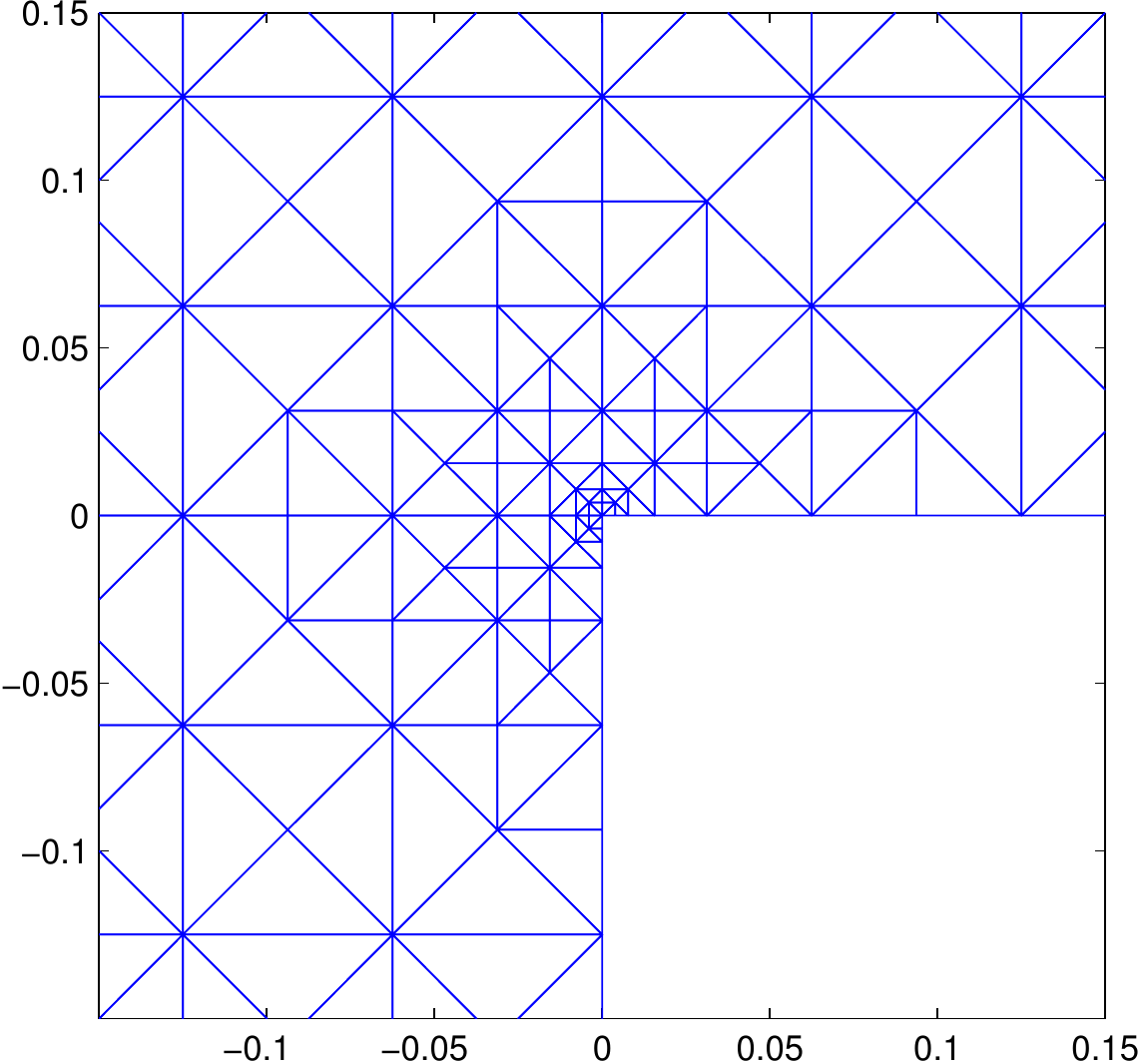}
  \end{minipage}\hspace*{\fill}
  \caption{\label{fig:meshes}Graded mesh with $\mu=0.3333$, generated
    by newest vertex bisection; left: whole mesh, right: zoom}
\end{figure}

The discretization errors for different mesh sizes and the
corresponding experimental orders of convergence are given in Tables
\ref{tab2a} and \ref{tab2b} for the interior angle $\omega=270^\circ$
and in Tables \ref{tab3a} and \ref{tab3b} for the interior angle
$\omega=355^\circ$.  We see that the numerical results confirm the
expected convergence rate $1/2$ for the dual singular complement
method and the finite element method on sufficently graded meshes.
For $\mu>2\lambda-1$ we obtain a convergence rate of about
$(\lambda-1/2)/\mu$ only which can certainly be proven with an
adaption of the techniques used in Section \ref{sec:graded} but is of
less interest. We show the numerical results here mainly to underline
that the strong grading $\mu<2\lambda-1$ is indeed necessary for
optimal convergence.

\begin{table}\centering\small
    \begin{tabular}{rcccc}
      \toprule
      unknowns & standard & eoc   & DSCM    & eoc   \\
      \midrule
      33          & 0.736    &       & 0.653 &       \\
      113         & 0.645    & 0.215 & 0.587 & 0.154 \\
      417         & 0.568    & 0.193 & 0.423 & 0.472 \\
      1601        & 0.503    & 0.181 & 0.303 & 0.482 \\
      6273        & 0.447    & 0.175 & 0.216 & 0.489 \\
      24833       & 0.397    & 0.171 & 0.154 & 0.493 \\
      98817       & 0.353    & 0.169 & 0.109 & 0.496 \\
      394241      & 0.314    & 0.168 & 0.077 & 0.498 \\
      \midrule
      expected    &          & 0.167 &       & 0.5 \\
      \bottomrule
    \end{tabular}\\[2ex]
    \caption{\label{tab2a}Discretization errors $e_h=y-y_h$ with 
      quasi-uniform mesh (standard) and $e_h=y-z_h$ (DSCM) for $\omega=270^\circ$}
\end{table}

\begin{table}\centering\small
    \begin{tabular}{rccrccrcc}
      \toprule
      \multicolumn{3}{c}{$\mu=0.666$} &
      \multicolumn{3}{c}{$\mu=0.5$}&
      \multicolumn{3}{c}{$\mu=0.333$} \\
      \cmidrule(r){1-3} \cmidrule{4-6} \cmidrule(l){7-9}
      unknowns& error & eoc &unknowns& error & eoc&unknowns& error & eoc \\
      \midrule
      33     & 0.736 &       & 33     & 0.736 &       & 33     & 0.736 &       \\
      113    & 0.645 & 0.215 & 113    & 0.645 & 0.215 & 113    & 0.645 & 0.215 \\
      421    & 0.498 & 0.392 & 424    & 0.505 & 0.369 & 428    & 0.445 & 0.559 \\
      1618   & 0.446 & 0.165 & 1631   & 0.398 & 0.354 & 1648   & 0.312 & 0.524 \\
      6343   & 0.348 & 0.361 & 6381   & 0.315 & 0.344 & 6463   & 0.220 & 0.512 \\
      25111  & 0.314 & 0.153 & 25244  & 0.249 & 0.339 & 25544  & 0.155 & 0.508 \\
      99881  & 0.246 & 0.354 & 100423 & 0.198 & 0.336 & 101563 & 0.110 & 0.504 \\
      398436 & 0.221 & 0.150 & 400553 & 0.157 & 0.335 & 405014 & 0.077 & 0.502 \\
       \midrule
      expected &     & 0.25  &        &       & 0.333 &        &       & 0.5   \\
      \bottomrule
    \end{tabular}
\caption{\label{tab2b}Discretization errors $e_h=y-y_h$ for $\omega=270^\circ$}
\end{table}

\begin{table}\centering\small
    \begin{tabular}{rcccc}
      \toprule
      unknowns & standard & eoc & DSCM & eoc \\
      \midrule
      46      & 1.105  &       & 1.010 &       \\
      159     & 1.069  & 0.053 & 1.021 &       \\
      589     & 1.049  & 0.029 & 0.834 & 0.291 \\
      2265    & 1.036  & 0.018 & 0.590 & 0.500 \\
      8881    & 1.028  & 0.012 & 0.417 & 0.500 \\
      35169   & 1.021  & 0.010 & 0.295 & 0.499 \\
      139969  & 1.015  & 0.008 & 0.209 & 0.497 \\
      558465  & 1.010  & 0.008 & 0.148 & 0.495 \\
      \midrule
      expected&        & 0.007 &       & 0.5 \\
      \bottomrule
    \end{tabular}\\[2ex]
\caption{\label{tab3a}Discretization errors $e_h=y-y_h$ with 
      quasi-uniform mesh (standard) and $e_h=y-z_h$ (DSCM) for $\omega=355^\circ$}
\end{table}

\begin{table}\centering\small
    \begin{tabular}{rccrccrcc}
      \toprule
      \multicolumn{3}{c}{$\mu=0.5$} &
      \multicolumn{3}{c}{$\mu=0.3$}&
      \multicolumn{3}{c}{$\mu=0.014085$} \\
      \cmidrule(r){1-3} \cmidrule{4-6} \cmidrule(l){7-9}
      unknowns& error & eoc &unknowns& error & eoc&unknowns& error & eoc \\
      \midrule
      46     & 1.105 &       & 46     & 1.105 &       & 46    & 1.105 &       \\
      159    & 1.069 & 0.053 & 159    & 1.069 & 0.053 & 159   & 1.069 & 0.047 \\
      597    & 1.039 & 0.044 & 602    & 1.031 & 0.055 & 970   & 0.854 & 0.325 \\
      2301   & 1.023 & 0.023 & 2335   & 1.012 & 0.028 & 4116  & 0.600 & 0.509 \\
      9014   & 1.011 & 0.017 & 9166   & 0.990 & 0.032 & 16154 & 0.424 & 0.502 \\
      35682  & 1.001 & 0.015 & 36197  & 0.975 & 0.022 & 62949 & 0.298 & 0.508 \\
      141986 & 0.991 & 0.015 & 144015 & 0.962 & 0.020 & 247276& 0.210 & 0.505 \\
      566419 & 0.981 & 0.014 & 574296 & 0.942 & 0.030 & 979316& 0.148 & 0.505 \\
      \midrule
      expected &     &0.014  &        &       & 0.023 &       &       & 0.5 \\
      \bottomrule
    \end{tabular}
\caption{\label{tab3b}Discretization errors $e_h=y-z_h$ for $\omega=355^\circ$}
\end{table}

\begin{figure}
  \hspace*{\fill}\includegraphics[width=0.33\textwidth]{JonnyDirichletControl/%
  y270mu0p05_3-crop}\hfill
  \includegraphics[width=0.33\textwidth]{JonnyDirichletControl/%
  yh270mu1p0_3-crop}\hspace*{\fill}

  \hspace*{\fill}exact solution \hfill quasi-uniform mesh \hspace*{\fill}\bigskip

  \hspace*{\fill}\includegraphics[width=0.33\textwidth]{JonnyDirichletControl/%
  yh270mu0p3333_3-crop}\hfill
  \includegraphics[width=0.33\textwidth]{JonnyDirichletControl/%
  zh270mu0p05_3-crop}\hspace*{\fill}
  
  \hspace*{\fill}graded mesh \hfill DSCM \hspace*{\fill}
  \caption{\label{fig:solutions}Visual comparison of the solutions, $\omega=270^\circ$}
\end{figure}

\begin{figure}
  \hspace*{\fill}\includegraphics[width=0.33\textwidth]{JonnyDirichletControl/%
  y355mu0p05_3-crop}\hfill
  \includegraphics[width=0.33\textwidth]{JonnyDirichletControl/%
  yh355mu1p0_3-crop}\hspace*{\fill}

  \hspace*{\fill}exact solution \hfill quasi-uniform mesh \hspace*{\fill}\bigskip

  \hspace*{\fill}\includegraphics[width=0.33\textwidth]{JonnyDirichletControl/%
  yh355mu0p0141_3-crop}\hfill
  \includegraphics[width=0.33\textwidth]{JonnyDirichletControl/%
  zh355mu0p05_3-crop}\hspace*{\fill}

  \hspace*{\fill}graded mesh \hfill DSCM \hspace*{\fill}
  \caption{\label{fig:solutions355}Visual comparison of the solutions, $\omega=355^\circ$}
\end{figure}

Finally, in Figures \ref{fig:solutions} and \ref{fig:solutions355} we
display the exact and some computed solutions for a visual comparison.
There is a pole of type $r^{-0.4999}$ in the boundary data and hence
in the exact solution. The standard finite element solution and the
solution on graded meshes are computed after regularization of the
boundary datum which replaces the infinite value for $r=0$ by a finite
one, which may be big as in the case of $\omega=355^\circ$. One can
also see that the behavior for $r\to0$ can be approximated better with
graded meshes. The solution with the DSCM contains two parts, a the
finite element function on a quasi-uniform mesh and a multiple of the
singular function $r^{-\lambda}\sin(\lambda\theta)$ which has a pole
of type $r^{-2/3}$ for $\omega=270^\circ$ and of type $r^{-180/355}$
for $\omega=355^\circ$. The latter term produces an infinite value for
$r=0$ and has a asymptotic behaviour which is different from the exact
solution. Interesting enough, the $L^2(\Omega)$-error of the DSCM
solution profits from the presence of this term.

Concerning the DSCM,  we emphasize that the quadrature formula
for the numerical evaluation of the integral
\[
  (u,\partial_n(r^\lambda\sin(\lambda\theta)))_\Gamma
\]
has to be adapted in order to get a sufficiently good approximation.
Otherwise, the error due to quadrature dominates the
overall error. In our implementation, we chose for the numerical
integration a graded mesh on the boundary ($h_E\sim hr_E^{1-\mu}$ if
the distance $r_E$ of the boundary edge $E$ satisfies $0<r_E<R$ with
$R$ being the radius of the refinement zone and $\mu$ being the
refinement parameter, and $h_T=h^{1/\mu}$ for $r_E=0$) combined with a
one-point Gauss quadrature rule on each element.  
%Furthermore, the grading parameter $\mu$ is chosen such that
The choice
$
  \mu\le2\pi/\omega-1
$
%which
seems to be the correct grading to achieve a convergence order
of $1/2$. For the results presented in Tables \ref{tab2a} and \ref{tab3a} we used
$R=0.1$ and $\mu=2\pi/\omega-1$.

\bibliographystyle{abbrv}\bibliography{ApPf}

\begin{thebibliography}{10}

\bibitem{ApelNicaisePfefferer2014b}
T.~Apel, S.~Nicaise, and J.~Pfefferer.
\newblock A dual singular complement method for the numerical solution of the
  {P}oisson equation with ${L}^2$ boundary data in non-convex domains.
\newblock Preprint arXiv:1505.00414 [math.NA], arXiv, 2015.

\bibitem{ApelNicaisePfefferer2014a}
T.~Apel, S.~Nicaise, and J.~Pfefferer.
\newblock Discretization of the {P}oisson equation with non-smooth data and
  emphasis on non-convex domains.
\newblock To appear in Numer. Methods Partial Differential Equations, 2016.

\bibitem{ApelPfeffererRoesch:2012}
T.~Apel, J.~Pfefferer, and A.~R\"osch.
\newblock Finite element error estimates on the boundary with application to
  optimal control.
\newblock {\em Math. Comp.}, 84:33--70, 2015.

\bibitem{Babuska1970}
I.~Babu{\v{s}}ka.
\newblock Error-bounds for finite element method.
\newblock {\em Numerische Mathematik}, 16:322--333, 1970/1971.

\bibitem{Berggren2004}
M.~Berggren.
\newblock {Approximations of very weak solutions to boundary-value problems.}
\newblock {\em SIAM J. Numer. Anal.}, 42(2):860--877, 2004.

\bibitem{Bernardi1989}
C.~Bernardi.
\newblock Optimal finite-element interpolation on curved domains.
\newblock {\em SIAM J. Numer. Anal.}, 26(5):1212--1240, 1989.

\bibitem{BlumDobrowolski1982}
H.~Blum and M.~Dobrowolski.
\newblock On finite element methods for elliptic equations on domains with
  corners.
\newblock {\em Computing}, 28(1):53--63, 1982.

\bibitem{BrambleKing1994}
J.~H. Bramble and J.~T. King.
\newblock A robust finite element method for nonhomogeneous {D}irichlet
  problems in domains with curved boundaries.
\newblock {\em Math. Comp.}, 63(207):1--17, 1994.

\bibitem{Carstensen1999}
C.~Carstensen.
\newblock {Quasi-interpolation and a posteriori error analysis in finite
  element methods.}
\newblock {\em M2AN, Math. Model. Numer. Anal.}, 33(6):1187--1202, 1999.

\bibitem{CasasMateosRaymond2009}
E.~Casas, M.~Mateos, and J.-P. Raymond.
\newblock Penalization of {D}irichlet optimal control problems.
\newblock {\em ESAIM. Control, Optimisation and Calculus of Variations},
  15(4):782--809, 2009.

\bibitem{CasasRaymond2006}
E.~Casas and J.-P. Raymond.
\newblock {Error estimates for the numerical approximation of Dirichlet
  boundary control for semilinear elliptic equations.}
\newblock {\em SIAM J. Control Optim.}, 45(5):1586--1611, 2006.

\bibitem{ciarletjr:03}
P.~Ciarlet, Jr. and J.~He.
\newblock The singular complement method for 2d scalar problems.
\newblock {\em C. R. Math. Acad. Sci. Paris}, 336(4):353--358, 2003.

\bibitem{CiarletJungKaddouriLabrunieZou2005}
P.~Ciarlet, Jr., B.~Jung, S.~Kaddouri, S.~Labrunie, and J.~Zou.
\newblock The {F}ourier singular complement method for the {P}oisson problem.
  {I}. {P}rismatic domains.
\newblock {\em Numer. Math.}, 101(3):423--450, 2005.

\bibitem{CiarletJungKaddouriLabrunieZou2006}
P.~Ciarlet, Jr., B.~Jung, S.~Kaddouri, S.~Labrunie, and J.~Zou.
\newblock The {F}ourier singular complement method for the {P}oisson problem.
  {II}. {A}xisymmetric domains.
\newblock {\em Numer. Math.}, 102(4):583--610, 2006.

\bibitem{Costabel1988}
M.~Costabel.
\newblock Boundary integral operators on {L}ipschitz domains: elementary
  results.
\newblock {\em SIAM J. Math. Anal.}, 19(3):613--626, 1988.

\bibitem{BourlardDaugeLubumaNicaise90a}
M.~Dauge, S.~Nicaise, M.~Bourlard, and J.~M.-S. Lubuma.
\newblock Coefficients des singularit\'es pour des probl\`emes aux limites
  elliptiques sur un domaine \`a points coniques. {I}.\ {R}\'esultats
  g\'en\'eraux pour le probl\`eme de {D}irichlet.
\newblock {\em RAIRO Mod\'el. Math. Anal. Num\'er.}, 24(1):27--52, 1990.

\bibitem{ReyesMeyerVexler2008}
J.~C. de~los Reyes, C.~Meyer, and B.~Vexler.
\newblock Finite element error analysis for state-constrained optimal control
  of the {S}tokes equations.
\newblock {\em Control Cybernet.}, 37(2):251--284, 2008.

\bibitem{DeckelnickGuentherHinze2009}
K.~Deckelnick, A.~G{\"u}nther, and M.~Hinze.
\newblock Finite element approximation of {D}irichlet boundary control for
  elliptic {PDE}s on two- and three-dimensional curved domains.
\newblock {\em SIAM J. Control Optim.}, 48(4):2798--2819, 2009.

\bibitem{Demlow:2010}
A.~Demlow, J.~Guzm\'{a}n, and A.~H. Schatz.
\newblock Local energy estimates for the finite element method on sharply
  varying grids.
\newblock {\em Math. Comp.}, 80(273):1--9, 2011.

\bibitem{FrenchKing1991}
D.~A. French and J.~T. King.
\newblock {Approximation of an elliptic control problem by the finite element
  method.}
\newblock {\em Numer. Funct. Anal. Optimization}, 12(3-4):299--314, 1991.

\bibitem{grisvard:85a}
P.~Grisvard.
\newblock {\em Elliptic problems in nonsmooth domains}, volume~24 of {\em
  Monographs and Studies in Mathematics}.
\newblock Pitman, Boston--London--Melbourne, 1985.

\bibitem{grisvard:92b}
P.~Grisvard.
\newblock {\em Singularities in boundary value problems}, volume~22 of {\em
  Research Notes in Applied Mathematics}.
\newblock Springer, New York, 1992.

\bibitem{LionsMagenes1968}
J.-L. Lions and E.~Magenes.
\newblock {\em {Probl\`emes aux limites non homog\`enes et applications. Vol.
  1, 2.}}
\newblock Travaux et Recherches Math\'ematiques. Dunod, Paris, 1968.

\bibitem{MayRannacherVexler2008}
S.~May, R.~Rannacher, and B.~Vexler.
\newblock Error analysis for a finite element approximation of elliptic
  {D}irichlet boundary control problems.
\newblock {\em SIAM J. Control Optim.}, 51(3):2585--2611, 2013.

\bibitem{OganesjanRuhovec1979}
L.~A. Oganesjan and L.~A. Ruhovec.
\newblock {\em Variatsionno-raznostnye metody resheniya ellipticheskikh
  uravnenii}.
\newblock Akad. Nauk Armyan. SSR, Erevan, 1979.

\bibitem{pfefferer:14}
J.~Pfefferer.
\newblock {\em Numerical analysis for elliptic Neumann boundary control
  problems on polygonal domains}.
\newblock PhD thesis, Universit\"at der Bundeswehr M\"unchen, 2014.
\newblock \url{http://athene.bibl.unibw-muenchen.de:8081/node?id=92055}.

\bibitem{Raugel1978}
G.~Raugel.
\newblock R\'esolution num\'erique par une m\'ethode d'\'el\'ements finis du
  probl\`eme de {D}irichlet pour le laplacien dans un polygone.
\newblock {\em C. R. Acad. Sci. Paris, S\'er. A}, 286(18):A791--A794, 1978.

\bibitem{StrangFix1973}
G.~Strang and G.~J. Fix.
\newblock {\em An analysis of the finite element method}.
\newblock Prentice-Hall, Inc., Englewood Cliffs, N. J., 1973.

\end{thebibliography}
\end{document}